\documentclass[11pt, reqno]{amsart}
\usepackage{amsmath}
\usepackage{amsthm}
\usepackage{amssymb}
\usepackage{amscd}
\usepackage{amsbsy}
\usepackage{epsfig}
\usepackage{graphicx}
\usepackage{psfrag}
\usepackage{bbm}
\usepackage{hyperref}
\usepackage{mathrsfs}
\usepackage[UKenglish]{babel}
\usepackage[UKenglish]{isodate}
\usepackage{cleveref}
\usepackage{todonotes}
\usepackage{soul}
\usepackage{xcolor}
\usepackage{enumitem}

\theoremstyle{plain}
\newtheorem{theorem}{Theorem}[section]

\newtheorem{remark}[theorem]{Remark}
\newtheorem{proposition}[theorem]{Proposition}
\newtheorem{lemma}[theorem]{Lemma}

\theoremstyle{definition}
\newtheorem{definition}[theorem]{Definition}

\newtheorem*{claim*}{Claim}
\newtheorem*{notation*}{Notation}

\newtheorem*{lemma*}{Lemma}

\theoremstyle{definition}

\def\R{\ensuremath{\mathbb R}}
\def\N{\ensuremath{\mathbb N}}

\def\U{\ensuremath{\mathcal U}}

\def\P{\ensuremath{\mathcal P}}

\def\F{\ensuremath{\mathcal F}}

\def\ie{{\em i.e.}, }

\def\E{\mathbb E}

\def\st{such that }

\def\ulx{\underline{x}}

\def\uly{\underline{y}}

\def\ulu{\underline{u}}
\def\i{\textbf{i}}

\numberwithin{equation}{section}

\begin{document}

\author[S.~Shaabanian]{Saeed shaabanian}
\thanks{I would like to thank my supervisor, Prof.\ Mike Todd, for their invaluable guidance, insightful comments, and continued support throughout this research.
I also thank Dr.\ Natalia Jurga for their thoughtful guidance and constructive feedback.}
\address{Mathematical Institute\\
University of St Andrews\\
North Haugh\\
St Andrews\\
KY16 9SS\\
Scotland} 
\email{ss507@st-andrews.ac.uk}
\urladdr{https://ss5071.wixsite.com/saeed-shaabanian}

\title{ Cover time for countable Markov shifts}
\date{\today}
\maketitle

\begin{abstract}
Cover time, in the context of dynamical systems, quantifies the rate at which orbits cover the system. We prove that for countable full shifts with a Gibbs measure, equipped with a natural metric, the rate of covering of orbits of points behaves according to the minimum measure of balls. Moreover, this rate exhibits sensitivity to changes in the metric.
\end{abstract}

\section{Introduction}\label{Sec:Int}

Assume that $(Y,d)$ is a compact metric space and let $X\subseteq Y$. Consider $(f,\mu)$ where $f:X\to X$ and $\mu$ is an ergodic probability measure.

For such a dynamical system, assume that $X=\overline{\{f^nx\}}_n$ for $\mu$ a.e. $x\in X$. We refer to this property as \emph{covering}. The question that arises is, for the point $x$, when does covering to some scale $\delta$ occur for the first time?
This question can be posed as follows.

\begin{definition}
    Let $\delta>0$. We say $A$ is \textit{$\delta$-dense} in $B$ if for all $ b\in B$, there exists $ a\in A$ such that $ d(a,b)\le \delta.$
\end{definition}

Cover time quantifies the rate at which orbits become dense in $X$:
\begin{definition}
    Let $\delta>0$. The \textit{$\delta$-covering time} is the function $\tau_\delta:X\to \N \cup \{\infty\}$ such that 
    \begin{equation}
        \tau_\delta(x):=\inf \{n\ge 1 : \{x, f(x),\ldots,f^n(x)\}\; \text{is} \; \delta\text{-dense in} \; X\}.
    \end{equation}
\end{definition}
A natural question here is, what is the asymptotic behaviour of the expected cover time in $(f,\mu)$ i.e. $\E_\mu (\tau_\delta)$ as $\delta \to 0$?

Cover time has been well studied in probability, such as in random-walk dynamics on graphs (\cite{Aldo1}, \cite{Lova} and \cite{DingLeePe}), Markov chains (\cite{LeviPer}), and stochastic geometry (\cite{FlaNew}, \cite{Janson}, \cite{Pesin} and \cite{Aldo2}). Recently, results have also been obtained in the chaos game using the cover time in \cite{JurgMorr, BarJurKol}. This concept was first studied by Jurga and Todd in the dynamical systems context \cite{MikeNatali}.
They proved for a large class of interval maps that $\E_\mu (\tau_\delta)$ is roughly like $\delta^{-\dim_M\mu}$  as $\delta\to 0$, where
\begin{equation}\label{Def: Minkowski}
\dim_M\mu:=\lim_{\delta\to0}\frac{\log M_{\mu}(\delta)}{\log\delta},\end{equation}
exists and is called the \textit{Minkowski dimension} of $\mu$  where $$M_{\mu}(\delta):=\min_{x\in\text{supp}(\mu)}\mu(B(x,\delta)).$$ The Minkowski dimension was first introduced in \cite{KennethJon} and by \cite[Proposition 4.2]{KennethJon} it is equal to the $L^{-\infty}$ dimension.

The sharp result in \cite{MikeNatali} is obtained for Gibbs Markov interval maps when the map has at least two full branches. In this case, \cite[Theorem 2.1]{MikeNatali} shows that $\E_{\mu} (\tau_\delta)$ is approximately like $ \delta^{-\dim_H\Lambda}\left(\log\frac{1}{\delta}\right)$, where $\dim_H\Lambda$ is the Hausdorff dimension of the repeller.
To obtain these results, they used the symbolic dynamics associated with interval maps. This observation motivated us to pursue this question for a symbolic system equipped with an appropriate metric. 

Assume that $(\Sigma,\sigma)$ is a countable full shift which admits a Gibbs measure. Since the usual metric here is not totally bounded and would yield $\E_{\mu} (\tau_\delta)=\infty$ for $\delta<1$, we use a metric with polynomial or exponential behaviour (see \eqref{Def:Metric d} and \eqref{Def:Metric d_2}). For such systems, we obtain the bounds for $\E_{\mu} (\tau_\delta)$. To this end, we first construct a minimal finite disjoint open cover of the space (see Proposition \ref{Lemma:Totally bounded in Full shift}), whose number of elements affects the limits of $\E_{\mu} (\tau_\delta)$, which helps us to obtain an upper bound for $\E_{\mu} (\tau_\delta)$.

One of the main tools here is the $\psi$-mixing property, which holds due to the fact that the measure is a Gibbs measure. This property can also play a key role through one of the standard recurrence rates in dynamical systems, namely the first hitting time.
\begin{definition}\label{Def: hitting time}
    Let $\delta>0$ and $U\subset \Sigma$. The \textit{first hitting time} of $\ulx\in \Sigma$ to $U$ is $\tau_U:\Sigma\to\N\cup\{\infty\}$ such that,
    $$\tau_U(\ulx)=\inf\left\{n\ge1 : \sigma^n(\ulx)\in U\right\}.$$
\end{definition}
Many results have been obtained regarding the behaviour of the hitting time. For background on hitting time statistics in dynamical systems, see \cite[Chapter 5]{Mike's book}, and for extensive information on hitting time in symbolic dynamical systems, see \cite{Coelho} and \cite{GalSch}. 

As a step in the proof of our main results, we also show that for $U$ a union of cylinders of the same depth, then $\E_\mu(\tau_U)$ is of order $1/\mu(U)$, see Theorem \ref{Lemma: upper bound for E(hitting time)}. The key issue here is that we require uniform constants, like in the case of intervals in \cite[Theorem 4.1]{MikeNatali}. We note that Kac Lemma says that $\E_{\mu_U}(\tau_U)=1/\mu(U)$, and our proof involves estimating $\E_\mu(\tau_U)$ in terms of $\E_{\mu_U}(\tau_U)$.

There is a close connection between hitting time and cover time. We use the mixing property of the system to exploit this relation, and then the problem of determining the covering behaviour of the system by orbits reduces to studying the hitting time behaviour on a finite open cover of the countable full shift. Then make use of some lemmas that were proved in \cite{MikeNatali} and that remain valid in our setting. These lemmas provide bounds for $\E_{\mu} (\tau_\delta)$ in terms of $\E_{\mu} (\tau_U)$ for suitable $U$. Then, using the estimated bounds obtained for $\E_{\mu} (\tau_U)$ and properties of the finite open cover, we derive the desired bounds for $\E_{\mu} (\tau_\delta)$.

Finally, we prove that for countable full shifts endowed with the Gibbs property, equipped with a suitable totally bounded metric, the rate of $\delta$-covering of orbits is variable and its range changes with respect to each metric. As we consider two metrics with different behaviours for the system in question: one exhibiting polynomial behaviour and the other exponential behaviour, it is shown that the expected time for the orbit of a point to cover the system equipped with the exponential metric is smaller than that for the same system equipped with polynomial behaviour. Note that, in this setting, obtaining a sharp lower bound, as done in  \cite[Theorem 2.1]{MikeNatali} (see Theorem \ref{Th:MN_sharp_theorem}), is out of reach, since the conditions for having a sharp lower bound in that theorem do not hold for a system with a Gibbs measure and the metrics under consideration (see Remark \ref{Remark: No Sharp Bound}).

\subsection{Layout of the paper.} Section \ref{Subsection: Cover time on Interval Maps} provides a brief review of the results obtained in \cite{MikeNatali}. In addition to presenting definitions that will be useful later, we restate the main results of that work. In Section \ref{Sec: CFS and main Thm}, we describe the countable full shift, examine its properties, and state our main theorems. Section \ref{Section: Propositions} is devoted to the results for the countable full shift equipped with a polynomial metric. We introduce a minimal finite open cover and derive an estimate for the first hitting time of the elements for this cover. Then we establish the proof of the results associated with this system. Section \ref{Section: exponential metric} is related to the results for the countable full shift with respect to an exponential metric. Here, we follow the same approach as in the previous section. Section \ref{Section: Examples} presents an example of this work.

\textbf{Notation}. We denote the cardinality of the set $A$ by $\#A$. We use $o(A_n)=B_n$, which means $\lim_n \frac{B_n}{A_n}=0$. So, $o(1)$ denotes a term which finally tends to zero. Also, $B_n=\mathcal{O}(A_n)$ means that there is some $C>0$ \st $B_n\le C A_n$. 

\section{Cover time for interval maps} \label{Subsection: Cover time on Interval Maps}
Let $I\subset \R$ be a bounded interval and $\mathcal{Z}=\{Z_i\}_{i\in \mathcal{I}}$ be at most a countable collection of subintervals of $I$ with disjoint interiors. Suppose that $f:\bigcup_{n\in \mathcal{I}}Z_n\to I$ is continuous and strictly monotone on each $Z_n$. Set $$\Lambda := \left\{x \in I : f^k(x) \in \bigcup_{n\in\mathcal{I}} Z_n \;\text{for all} \;k\ge 0 \right\}.$$

 Define a nonatomic Borel probability measure $m$ such that $m(D)=0 $ where $D=I\setminus \bigcup_i(Z_i)$, which is conformal with respect to a potential $\phi:\Lambda \to \R$ (i.e. $\frac{dm}{d(m\circ f)}=e^{\phi}$).

One of the most important features of $f$ is its symbolic coding, which will play a key role for us later on. Let $\mathcal{Z}^n$ be the set $n$-cylinders, i.e. this is the set of maximal intervals $Z$ of $I$ \st $f^k(x)\subset Z_{i_k}$ for some $Z_{i_k}\in \mathcal{Z}$ for $k\in \N_0$. Let $\Sigma' \subset \mathcal{I}^{\N}$ be the subshift given by the set of sequences $\textbf{i}=(i_0,i_1,\ldots)\in \mathcal{I}^{\N_0}$ for which there exists an $x\in \Lambda$ such that $f^n(x)\in Z_{i_n}$ for all $n\in \N_0$. Define the projection $\Pi:\Sigma'\to I$ such that,
$$\Pi(i_0,i_1,\ldots)=\lim_{n\to\infty}f^{-1}_{i_0} \circ \cdots \circ f^{-1}_{i_{n-1}}(I)=x,$$ where $f_i=f|_{Z_i}$. Then $f\circ \Pi=\Pi \circ \sigma$ where $\sigma: \Sigma' \to \Sigma'$ is the left shift map.
Denote $\tilde{\mu}$ to be the measure on $\Sigma'$ such that $\mu=\Pi_*\tilde{\mu}$, $\Sigma'_n$ the set of all words of length $n$ in $\Sigma'$ and $\Sigma'^*$ the set of all finite words in $\Sigma'$. Also for $w\in\Sigma'^*$ we write $[w]=\{\textbf{i}\in\Sigma': \textbf{i}|n=w\}$ where $\textbf{i}|n=(i_0,\ldots,i_{n-1})$ if $\textbf{i}=(i_n)_{n\in\N_0}.$

The following two assumptions for $\tilde{\mu}$ are used as conditions of the main theorems in \cite{MikeNatali}. First $\tilde{\mu}$ is quasi-Bernoulli i.e. there exists $C_*>1$ such that for all finite words $\textbf{i}, \textbf{j} \in \Sigma'^*, C_*^{-1}\tilde{\mu}([\textbf{i}])\tilde{\mu}([\textbf{j}])\le \tilde{\mu}([\textbf{i}\textbf{j}])\le C_*\tilde{\mu}([\textbf{i}])\tilde{\mu}([\textbf{j}])$, and the second assumption is $\tilde{\mu}$ is $\psi$-mixing, i.e. 
\begin{equation}\label{Equ:psi mixing}
    \left|\frac{\tilde{\mu}(\cup[x_0,\ldots,x_{i-1},y_0,\ldots,y_{j-1},z_0,\ldots,z_{l-1}]}{\tilde{\mu}([x_0,\ldots,x_{i-1}])\tilde{\mu}([z_0,\ldots,z_{l-1}])}-1\right|\le \psi(j),
\end{equation}
where the union is taken over all words $y_0, \ldots, y_{j-1}$ of length $j$ such that $$(x_0,\ldots,x_{i-1},y_0,\ldots,y_{j-1},z_0,\ldots,z_{l-1})\in\Sigma'_{i+j+l},$$ and $\psi(j)\to 0$ as $j\to \infty.$ 

We introduce a class of systems that satisfy the conditions of the theorems in \cite{MikeNatali} (see Remark\ref{Remark:Cond}).

\begin{definition}\label{Def:Gibbs}
   Given a Lipschitz potential $\phi:\Sigma\to \R$, $\mu$ (or $\tilde{\mu}$)  is \textit{Gibbs} if there exist constants $K,P$ such that for all $i_0,\ldots,i_{n-1}\in\Sigma'_n$ and $\textbf{i}\in[i_0,\ldots,i_{n-1}],$ $$ K^{-1}\le\frac{\tilde{\mu}([i_0,\ldots,i_{n-1}])}{e^{S_n\phi(\Pi(\textbf{i}))+nP}}\le K,$$
    where $S_n\phi(x)=\sum_{i=0}^{n-1} \phi \circ f^i(x).$
\end{definition}

\begin{definition}\label{Def:BIP}
    $f:\Lambda\to\Lambda$ is  \textit{Markov}, if for each $Z\in\mathcal{Z}, f(Z)$ is a union of elements of $\mathcal{Z}$. Also, we say that $f$ satisfies the \textit{big images and pre-images property (BIP)} if there exists a finite set $\{Z_j\}_{j\in\mathcal{J}}\subset\mathcal{Z}$ such that,
    $$\forall Z\in\mathcal{Z}, \exists j,k\in \mathcal{J} : Z\subseteq f(Z_j),Z_k\subseteq f(Z).$$

\end{definition}

\begin{definition}
    We say that $(f,\mu)$ is \textit{Gibbs-Markov} with BIP if $f$ is Markov and satisfies BIP and $\mu$ is Gibbs.
\end{definition}

 In the following a condition of the existence of a pairwise disjoint finite open cover is used in \cite{MikeNatali}. We mention it here since to obtain the result under discussion in this note we use an analogous cover.

\label{Condition (U)}\textbf{Condition (U).} We say that the system $(f,\mu)$ satisfies \textbf{(U)} if there exists $\delta_0>0$ such that for all $0<\delta\le \delta_0$ we can find a finite collection $\mathcal{U}_{\delta}$ of cylinders such that

(U)(a) there exists $0<t<1<T$ such that for all elements $U$ in $\mathcal{U}(\delta_0):=\bigcup_{0<\delta\le \delta_0}\bigcup\left\{U:U\in \mathcal{U}_{\delta}\right\} $, there exists $x\in\Lambda$ such that $B(x,t\delta)\subseteq  U \subseteq B(x,T\delta)$,

(U)(b) for all $0<\delta\le \delta_0$ the cylinders in $\mathcal{U}_{\delta}$ are pairwise disjoint,

(U)(c) for all $0<\delta\le \delta_0$, $\Lambda$ is covered by $\bigcup_{U\in\mathcal{U}_{\delta}}U$,

(U)(d) for all $U\in \mathcal{U}_{\delta}$, $U\cap\Lambda= \bigcup_{\textbf{i}\in\P_{\delta}}[\textbf{i}]$ where $\P_{\delta}\subset\Sigma^*$ is an at most countable collection of words \textbf{i} with the property that $|\textbf{i}|=\mathcal{O}(\log(\frac{1}{\delta}))$ (where the implied constant is uniform over all $\delta$).

Now, the main theorem in \cite{MikeNatali} can be stated for $(f,\mu)$.
\begin{theorem}\label{Th:MN_main_theorem}
    \cite[Theorem 2.2]{MikeNatali} Assume that $(f,\mu)$ satisfies condition \textbf{(U)} and some other conditions (described in Remark \ref{Remark:Cond}), then there exist $0<c<C<\infty$ and $\varepsilon>0$ such that for all $\delta>0$,
    $$\frac{c}{M_\mu(\frac{\delta}{\varepsilon})}\le\E_\mu(\tau_\delta)\le \frac{C}{M_\mu(\varepsilon \delta)}\left(\log\frac{1}{\delta}\right).$$
    In particular if $\dim_M\mu<\infty$ then
    $$c\delta^{-\dim_M\mu+\text{Err}(\frac{\delta}{\varepsilon})}\le \E_\mu(\tau_\delta)\le C\delta^{-\dim_M\mu-\text{Err}(\varepsilon\delta)}\left(\log\frac{1}{\delta}\right),$$
    where $\text{Err}(\delta):=\left|\dim_M\mu-\frac{\log(M_\mu(\delta))}{\log\delta}\right|.$
\end{theorem}

\begin{remark}\label{Remark:Cond}
\begin{enumerate}
    \item Three conditions relate to the regularity properties of $\phi$. Two conditions are that $\tilde{\mu}$ satisfies Definitions \ref{Def:Gibbs} and \ref{Def:BIP}, and the other conditions, e.g. control over quantities such as ratio of derivatives of iterates of $f$, ratio of adjacent cylinders which all can be found in the \cite{MikeNatali}, and we have omitted here for brevity.
   \item The class of Markov maps satisfying the regularity properties of $\phi$, Definitions \ref{Def:Gibbs} and \ref{Def:BIP}, and also condition (a4) mentioned in \cite[p4]{MikeNatali}, includes the class of Gibbs-Markov maps with BIP.
    \end{enumerate}
\end{remark}

The following property of the measure helps to obtain a sharp result. The measure $\mu$ is \textit{Ahlfors regular}, if there exists $c>0$ \st for all $x\in \Lambda$ and $r>0$, $$\frac{r^{\dim_H\Lambda}}{c}\le\mu(B_r(x))\le cr^{\dim_H\Lambda},$$  where $\dim_H\Lambda$ is the Hausdorff dimension of $\Lambda$. Note that the Lebesgue measure has this property on intervals or, more generally, a conformal measure, i.e., a measure which satisfies condition $\frac{d(\mu \circ f)}{d\mu}=e^{-\phi}$ with potential $\left(-\dim_H\Lambda\right)\log \left|D_f\right|$, may satisfy this.

\begin{theorem}\label{Th:MN_sharp_theorem}
    \cite[Theorem 2.1]{MikeNatali} Assume that $(f,\mu)$ satisfies condition \textbf{(U)} and some other conditions (described in Remark \ref{Remark:Cond}), and also $\mu$ is \textit{Ahlfors regular}, then there exist $0<c<C<\infty$ such that for all $\delta>0$,
    $$c \delta^{{-\dim_H\Lambda}} \le\E_\mu(\tau_\delta)\le C \delta^{{-\dim_H\Lambda}} \left(\log\frac1{\delta}\right).$$
   Moreover, if the system is Gibbs Markov and $f$ has at least 2 full branches, then we also have the sharp lower bound,
    $$c \delta^{{-\dim_H\Lambda}} \left(\log\frac1{\delta}\right)\le\E_\mu(\tau_\delta)\le C \delta^{{-\dim_H\Lambda}} \left(\log\frac1{\delta}\right).$$
  
\end{theorem}

\section{Setting and main results}\label{Sec: CFS and main Thm}
Now, we turn to the dynamical system that will be our study's focus.

We define the symbolic space by 
$$\Sigma:=\left\{\ulx=(x_0,x_1,\ldots)\in\N^{\N_0}:a_{x_nx_n+1}=1,\ \forall n\in {\N_0}\right\},$$
for $A=(a_{ij})$, an $\N \times\N$ matrix of $0$s and $1$s, where $\N_0=\N\cup\{0\}$. The shift map $\sigma:\Sigma\to\Sigma$ is defined by $\sigma(\ulx)=(x_1, x_2, \ldots)$, and the system $(\Sigma,\sigma)$ is called a \textit{one-sided countable Markov shift} (CMS). If all the entries of the matrix are one, then $\Sigma$ includes all sequences over the $\N$, and it is called a \textit{one-sided countable full shift} (CFS).

For $n\in\N$, the subset $$Z_n=[y_0,\ldots,y_{n-1}]:=\left\{\ulx\in\Sigma : y_0=x_0,\ldots,y_{n-1}=x_{n-1}\right\}\subset\Sigma,$$ is called an \textit{$n$-cylinder}. Denote the collection of all $n$-cylinders by $\mathcal{Z}^n$, then $\Sigma$ has a topology generated by the cylinder sets \ie $\left\{Z_n:Z_n\in \mathcal{Z}^n, \forall n\in\N\right\}$.

We denote the set of all words in $\Sigma$ with length $n$ by $\Sigma_n$ and $\Sigma^*:=\cup_{k\ge 0}\Sigma_k$ denotes the set of all finite words in $\Sigma$. 

Note that depending on $A$, $\Sigma$ may be locally compact \cite[Observation 7.2.3]{Kitchens}, but the system $(\Sigma,\sigma)$ is never compact. However, for the cover time definition to make sense, the system must be a subspace of a compact space. In, for example \cite{Zargar} and \cite[Section 4.1]{IommiTodd}, by defining a metric associated with a totally bounded metric, a completed space can be constructed for CMS that is compact. As this metric is employed throughout this paper, we begin by introducing it.

Recall that a metric space $(X,\rho)$ is totally bounded if for all $\varepsilon>0,$ there exists a finite set $\{x_1,\ldots,x_k: x_i\in X, \forall\ 1\le i\le k \}$ \st $X\subset\bigcup_{i=1}^k B_{\varepsilon}(x_i)$.

\begin{remark}
   Define $\rho:\N\times\N\to\R$ by $$\lim_{n\to\infty}\sup_{i,j\ge n} \rho(i,j)=0,$$ then $(\N,\rho)$ is a totally bounded metric space. Therefore, one can consider totally bounded metrics with different asymptotic behaviours. An example of a polynomial behaviour of a totally bounded metric on $\N$ is $\rho_1(x, y) = \left| \frac{1}{x} - \frac{1}{y} \right|,$ which was used in \cite{Zargar}. Also, the metric $\rho_2(x, y) = \left| \frac{1}{e^{\alpha x}} - \frac{1}{e^{\alpha y}} \right|,$ for some $\alpha>0$ exhibits exponential behaviour. The discrete metric is not totally bounded.
\end{remark}

Now, let $\theta\in(0,1)$, and $\rho:\N\times\N\to[0,1]$ be a totally bounded metric, then 
$$ d(\ulx,\underline{y})=\sum_{n\ge 0} \theta^n \rho(x_n,y_n),$$ is a metric on $\Sigma$ with the following properties (see e.g. \cite[Section 4.1]{IommiTodd}):

\begin{itemize}
    \item $(\overline{\Sigma}_\rho,\overline{d}
)$, the completion of $\Sigma$ with respect to totally bounded metric $\rho$, is compact,
    \item $d$ generates the cylinder topology and $\sigma$ is a uniformly continuous on $\Sigma$ where it extends to $\overline{\Sigma}$ as a continuous map.
\end{itemize}

    One of the key features of the metric $d$ is that it is totally bounded. How this result is obtained is shown in Proposition \ref{Lemma:Totally bounded in Full shift}.

We next explain our dynamics and measures.

\begin{definition}
    Given $\phi:\Sigma \to \R$ be the potential on $\Sigma$, the \textit{$n$-th variation of $\phi$} is,
    $$ var_n(\phi):=\sup \left\{|\phi(x)-\phi(y)|:x_0=y_0,\ldots,x_{n-1}=y_{n-1}\right\}.$$
\end{definition}
We say $\phi$ has \textit{summable variations} if $\sum_{n\ge 2} var_n(\phi)<\infty$.

Similarly, Definition \ref{Def:BIP} for countable Markov shifts can be rewritten as follows.
\begin{definition}\label{Def:BIP_CMS}
    Let $A$ be the transition matrix. $A$ has the \textit{big images property (BIP)} if
    $$ \exists\ b_1,\ldots,b_N \in \N : \forall t\in\N, \quad \exists\ i,j \quad \text{such that}  \quad a_{b_it}a_{tb_j}=1.$$
\end{definition}
Note that CFS naturally has this condition.  

\begin{definition}
    Let $\phi$ be the potential on a countable full shift. The \textit{Gurevich pressure} of $\phi$ is, 
    $$ P_G(\phi)=\lim_{n\to\infty} \frac{1}{n}\log\sum_{\sigma^n \underline{x}=\underline{x}}e^{S_n\phi(\underline{x})} .$$
\end{definition}

Note that it can be shown that Gurevich pressure exists \cite{Sarig1}. We can rewrite Definition \ref{Def:Gibbs} for countable Markov shifts.

\begin{definition}\label{Definition:Gibbs measure in CMS}
A \textit{Gibbs measure} for $\phi$ on a countable Markov shift is an invariant probability measure $m$ such that for some constants $K>1$ and $P$ and every $Z_n$,
$$K^{-1}\le\frac{m(Z_n)}{e^{S_n\phi(\underline{x})-nP}}\le K, \quad \forall \underline{x} \in Z_n.$$
\end{definition}
Note that by \cite[Proposition 2.2]{MauldinUrbanski01}, $P$ equals the pressure $P_G(\phi)$. For convenience in calculations, we usually normalise the potential by $\phi-P_G(\phi)$, which ensures that $P=0$.

From this point onwards, we assume $(\Sigma,\sigma)$ represents the countable full shift space and $\phi$ be the potential \st $\phi\big|_{[n]}\le -\kappa\log (n+1)$ for some $\kappa>1$, then $\phi$ has Gurevich pressure and by \cite[Theorem 1]{Sarig2}, implies that $(\Sigma,\sigma)$ has a Gibbs measure. It is well known that any Gibbs measure is a $\psi$-mixing measure as well (see \cite[Lemma 2.4]{MelbournNicol05}). 

We consider the following metrics and, endowing the CFS with each of them, examine the asymptotic behaviour of $\E(\tau_\delta)$,

\begin{equation}\label{Def:Metric d}
d_1(\ulx,\underline{y}):=\sum_{n\ge 0} \theta^n \left| \frac{1}{x_n} - \frac{1}{y_n} \right|,\end{equation}
\begin{equation}\label{Def:Metric d_2}
    d_2(\ulx,\underline{y}):=\sum_{n\ge 0} \theta^n  \left| \frac{1}{e^{\alpha x_n}} - \frac{1}{e^{\alpha y_n}} \right|,
\end{equation}where $\alpha>0$, $\theta\in(0,1)$ and $\ulx,\uly$ are any elements that belong to the full shift.

Now, we can state our main results.
\begin{theorem}\label{Theorem: main theorem}
   Assume that $(\Sigma,\sigma)$ is the countable full shift equipped with the metric $d_1$ as in \eqref{Def:Metric d}. Let $\phi:\Sigma\to\R$ be a potential \st $\phi\big|_{[n]}\le -\kappa\log (n+1)$ for some $\kappa>1$, which admits the Gibbs measure  $\mu$. Then there exist $0<c<C<\infty$ and $\varepsilon>0$ such that for all $\delta>0$,
$$\frac{c}{M_\mu(\delta/\varepsilon)}\le\E(\tau_\delta)\le \frac{C}{M_\mu(\varepsilon\delta)}\left(\log\frac1{\delta}\right)^2. $$
\end{theorem}



By replacing the metric with $d_2$, we obtain:
\begin{theorem}\label{Theorem: main theorem 2}
   Assume that $(\Sigma,\sigma)$ is the countable full shift equipped with the metric $d_2$ as in \eqref{Def:Metric d_2}. Let $\phi:\Sigma\to\R$ be a potential \st $\phi\big|_{[n]}\le -\kappa\log (n+1)$ for some $\kappa>1$, which admits the Gibbs measure $\mu$. Then there exist $0<c<C<\infty$ and $\varepsilon>0$ such that for all $\delta>0$, we have
$$\frac{c}{M_\mu(\delta/\varepsilon)}\le\E(\tau_\delta)\le \frac{C}{M_\mu(\varepsilon\delta)}\left(\log\frac1{\delta}\right)\left(\log\left(\log\frac1{\delta}\right)\right).$$

\end{theorem}

\begin{remark}\label{Remark: Conclusion after Main Thm}
\begin{enumerate}
    \item In these results, the upper bounds depend on the metric. This is due to the use of a finite open cover on $(\Sigma,\sigma)$ to derive these upper bounds, since the cardinality of the cover directly affects the bounds. In fact, the cover is chosen with properties that allow us to reduce the problem of estimating the expected cover time to that of estimating the expected hitting time of the elements of the cover. Then, these properties also allow us to derive an upper bound for this expected hitting time in terms of the number of sets in the cover. In each metric setting, the cardinality of the finite open cover is compared in exact correspondence with the bounds discussed above.
    \item In view of the obtained results and those established in \cite{MikeNatali}, we conclude that the upper bound of the cover time is sensitive to changes in the metric. Moreover, the bounds obtained for the metric $d_1$ are larger than those corresponding to the metric $d_2$, and both exceed the bounds for the interval maps (i.e. results in \cite{MikeNatali}). That is, one expects that the orbits in an interval map cover the system rapidly in comparison to the orbits in a CFS with metrics $d_1$ and $d_2$. Furthermore, the orbits in a CFS equipped with a metric with exponential behaviour cover the system faster than when it is equipped with a polynomial metric, but still slower than the interval maps.
\end{enumerate}

\end{remark}

\begin{lemma}\label{Lemma: No Ahlfors exists}
     Assume that $(\Sigma,\sigma)$ is the countable full shift equipped with the metric $d_1$ as in \eqref{Def:Metric d}. Let $\phi:\Sigma\to\R$ be a potential \st $\phi\big|_{[n]}\to-\infty$ as $n\to \infty$, which admits the Gibbs measure  $\mu$. Then, $\dim_M\mu=\infty$. Consequently, Ahlfors regularity does not hold for $\mu$.
\end{lemma}
\begin{proof}
    We first evaluate the diameter of the cylinders with respect to $d_1$. The greatest distance between elements of $Z_k=[u_0,\ldots,u_{k-1}]\in \mathcal{Z}^k$ is approached for $\ulx$ and $\uly$ when they take the following forms:
$$\ulx=(u_0,\ldots,u_{k-1},\!\!\!\!\!\!\!\!\!\!\!\!\stackrel{\scriptstyle k\text{-th component}}{\stackrel{\small{\downarrow}}M}\!\!\!\!\!\!\!\!\!\!\!\!,M,\ldots),$$
$$\uly=(u_0,\ldots,u_{k-1},1,1,\ldots),$$
where $M$ is a large number. Then we have, $$ d_1(\ulx,\underline{y})=\theta^k\left(\sum_{n\ge 0} \theta^n \left| \frac{1}{M} - 1 \right|\right),$$
and as $M\to \infty$, $d_1(\ulx,\underline{y})=\frac{\theta^k}{1-\theta}$. Hence,
\begin{equation}\label{Diam_Cylinder_Full}
    \text{diam}(Z_k)=\frac{\theta^k}{1-\theta}. 
\end{equation}

Now, let $\delta>0$. If $\text{diam}(Z_k)=\frac{\theta^k}{1-\theta}=\delta$, then a cylinder with the length of $k=\left\lceil\frac{\log(\delta(1-\theta))}{\log\theta}\right\rceil$ has diameter $\delta$. This yields for all $\ulx\in \Sigma$, $\mu\left(B_\delta(\ulx)\right)=\mu \left([x_0,\ldots,x_{k-1}]\right)$, where $k=\left\lceil\frac{\log(\delta(1-\theta))}{\log\theta}\right\rceil$. To show $\dim_M\mu=\infty$, we first approximate the quantity $M_\mu (\delta)=\min_{\ulx}\mu\left(B_\delta(\ulx)\right)$ using the Gibbs property, and then use this approximation to compute $\dim_M\mu$. By the Gibbs property, there exists some $K>1$ \st 
$$\mu\left(B_\delta(\ulx)\right)=\mu (\left[x_0,\ldots,x_{k-1}]\right)\le K e^{S_k\phi(\ulx)}= K \prod_{i=0}^{k-1}e^{\phi(\sigma^i(\ulx))}\le K \left(e^{\sup_{{x\in[x_i], 0\le i\le k-1}}(\phi(x))}\right)^k,$$
it follows that \begin{align*}
\min_{\ulx\in \Sigma}\left(\mu\left(B_\delta(\ulx)\right)\right)&\le K \inf_{\ulx\in \Sigma}\left(e^{\sup_{{x\in[x_i], 0\le i\le k-1}}(\phi(x))}\right)^k\le K \left(\inf_{\ulx\in \Sigma}e^{\sup_{{x\in[x_i], 0\le i\le k-1}}(\phi(x))}\right)^k\\&= K \left(e^{\inf_{\ulx\in \Sigma}\left(\sup_{{x\in[x_i], 0\le i\le k-1}}(\phi(x))\right)}\right)^k =K \left(e^{\limsup_{x\in [n],n\in\N}\left(\phi(x)\right)}\right)^k.\end{align*}
Therefore,
$$\log \left(\min_{\ulx\in \Sigma}\left(\mu\left(B_\delta(\ulx)\right)\right)\right)\le \log\left(K\left(e^{\limsup_{x\in [n],n\in\N}\left(\phi(x)\right)}\right)^k \right), $$
after dividing by $\log(\delta)$ and taking $\lim_{\delta\to0}$ on both sides, we obtain 
\begin{align*}
    \dim_M\mu&=\lim_{\delta\to0}\frac{\log M_{\mu}(\delta)}{\log\delta}\ge \lim_{\delta\to0}\frac{\log \left(K\left(e^{\limsup_{x\in [n],n\in\N}\left(\phi(x)\right)}\right)^k\right)}{\log\delta}\\&=\lim_{\delta\to0}\frac{\log K}{\log\delta}+ \lim_{\delta \to 0}\frac{\frac{\log(\delta(1-\theta))}{\log\theta} \left(\log e^{\limsup_{x\in [n],n\in\N}\left(\phi(x)\right)}\right)}{\log \delta}\\& =\lim_{\delta \to 0}\frac{\frac{\log\delta}{\log\theta} \left({\limsup_{x\in [n],n\in\N}\left(\phi(x)\right)}\right)}{\log \delta}+\lim_{\delta \to 0}\frac{\frac{\log(1-\theta)}{\log\theta} \left( {\limsup_{x\in [n],n\in\N}\left(\phi(x)\right)}\right)}{\log \delta}\\&=\frac{{\limsup_{x\in [n],n\in\N}\left(\phi(x)\right)}}{\log \theta}. 
\end{align*}
Since $\phi$ on $[n]$ tends to $-\infty$  as $n\to \infty$, we obtain $\dim_M\mu= \infty$.

Since there is no chance to get $\min_{\ulx\in \Sigma}\left(\mu\left(B_\delta(\ulx)\right)\right)\le \delta^{c}$ for some constant $c>0$, the Ahlfors regularity condition also fails to hold.
\end{proof}
\begin{remark}\label{Remark: No Ahlfors exists for d_2}
    An analogous result to Lemma \ref{Lemma: No Ahlfors exists} holds for $(\Sigma,\sigma)$ with respect to metric $d_2$.
\end{remark}
\begin{remark}\label{Remark: No Sharp Bound}
    In view of Lemma \ref{Lemma: No Ahlfors exists} and Remark \ref{Remark: No Ahlfors exists for d_2}, an analogous statement to Theorem \ref{Th:MN_sharp_theorem} and the second part of Theorem \ref{Th:MN_main_theorem} fail to hold in our setting. In fact, regarding the lower bound obtained under Ahlfors regularity in \cite{MikeNatali}, they use certain techniques (see \cite[Section 5]{MikeNatali}) to derive $$c\min_{U\in\U_\delta}\frac{1}{\mu(U)}\log\left(\frac{1}{\delta}\right)\le \E(\tau_\delta), $$ for some $c\in(0,1)$, and then use Ahlfors regularity to get the sharp lower bounds (Theorem \ref{Th:MN_sharp_theorem}). In our setting, we are able to obtain $c\min_{U\in\U_\delta}\frac{1}{\mu(U)}(\#\U_\delta)\le \E(\tau_\delta),$ for some open cover $\U_\delta$ satisfying condition (\textbf{U}) with respect to the metrics $d_1$ or $d_2$. However, the above bound is worse than $\frac{c}{M_\mu\left(\frac{\delta}{\varepsilon}\right)}$.  
\end{remark}

\section{Cover time under a polynomial metric}\label{Section: Propositions}

In this section, we prove the result related to the metric $d_1$, that is, Theorem \ref{Theorem: main theorem}. First, we present one proposition and one theorem that play a key role in the proof of our result. The first establishes a finite open cover satisfying condition \hyperref[Condition (U)]{\textbf{(U)}} for the full shift with respect to metrics $d_1$, the second introduces a connection between cover time and hitting time.

We divide this section into two subsections. First, we present the propositions mentioned above and in the second subsection, we prove Theorem \ref{Theorem: main theorem}. 

\subsection{Preliminary Propositions for $d_1$}\label{subsection: PP for d_1}

We start with a proposition which gives a pairwise disjoint cover of the union of cylinders on $\Sigma$ with respect to the metric $d_1$, which satisfies the condition \hyperref[Condition (U)]{\textbf{(U)}}.  
\begin{proposition}\label{Lemma:Totally bounded in Full shift}
    Let $\delta>0$ and assume that $(\Sigma,\sigma)$ is a countable full shift with the metric $d_1$ as in \eqref{Def:Metric d}. Then, there exists a pairwise disjoint finite open cover $\U_{\delta}$ for $(\Sigma,d_1)$ with $\#(\U_{\delta})=\left\lceil\left(2\sqrt{\frac{1}{\delta}}\left(1+o\left(\sqrt{\delta}\right)\right)\right) ^ { \frac{\log(\delta(1-\theta))}{\log\theta}}\right\rceil$ \st
\begin{enumerate}
    \item for all $U\in \U_{\delta}$, we have $U=\bigcup_{\mathbf{i}\in \mathrm{U^*}}[\mathbf{i}]$ \st $|\mathbf{i}|= \left\lceil\frac{\log(\delta(1-\theta))}{\log\theta}\right\rceil$, where $\mathrm{U^*}\subset\Sigma^*$ is a set of finite words of length $|\mathbf{i}|$ with respect to $U$,
    \item there exists $1<T<\infty$ \st for all $\ulx\in \Sigma$, there exists $U\in\U_\delta$ \st $B_{\delta}(\ulx)\subset U\subset B_{T\delta}(\ulx)$.
\end{enumerate}

\end{proposition}
\begin{proof}
    
We divide the proof into two steps. In the first step, we are looking for a finite collection of cylinders that efficiently cover $\Sigma$. In the second step, we then verify properties (1) and (2) for this collection.

\textit{Step 1}. To obtain the desired cover, the idea is to take advantage of the total boundedness of $\rho_1$, since this allows us to find a finite cover of $\N$ with respect to $\rho_1$. For this reason, let $\delta> 0$, and start by letting $\mathcal{V}_\delta$ be the empty set, to which elements are then added step by step. Let $N$ be the minimal natural number  \st $\frac{1}{N}\le \delta$. In fact, $N$ provides a convenient approximation of the diameter of the neighbourhoods, which simplifies the computations. Moreover, for
\begin{align}\label{Countable set of N}
    \mathcal{N}_{\rho_1}(N):=\{N,N+1,\ldots\},
\end{align}
if $n\in\mathcal{N}_{\rho_1}(N)$, then ${\rho_1}(N,n)=\frac{1}{N}-\frac{1}{n}<\delta$. We add $\mathcal{N}_{\rho_1}(N)$ to $\mathcal{V}_\delta$.

 Now, we need to find neighbourhoods that cover all $n<N$ for the desired finite open cover. We divide the argument into two parts. Some of the smallest natural numbers have the property that their $\frac{1}{N}$-neighbourhood contains no other natural numbers. We first identify them and then discuss the remaining elements.
 
We can obtain the largest number $n<N$ \st $\mathcal{N}_{\rho_1}(n)$ is a singleton using the following inequality 
$$\frac{1}{n-1}-\frac{1}{n}<\frac{1}{N},$$
after simplification, we have
$n^2-n-N>0$, which admits the following solutions
$$n>\frac{1+ \sqrt{1+4N}}{2}\ge \frac{1+ \sqrt{1+\frac{4}{\delta}}}{2},$$

therefore, $N_s:=\left\lfloor\frac{1+ \sqrt{1+\frac{4}{\delta}}}{2}\right\rfloor$ is the largest $n<N$ with a singleton neighbourhood. We add all $\mathcal{N}_{\rho_1}(n)$ for $n<N_s$ to $\mathcal{V}_\delta$.

To obtain an optimal finite open cover for $\N$, we are looking for $\delta$-neighbourhoods for all $n$, \st $N_s<n<N$. To this end, let $\mathcal{N}_{\rho_1}(N-1)$ be the last neighbourhood before $\mathcal{N}_{\rho_1}(N)$ whose elements satisfy: 
$${\rho_1}(N-1-i,N-1)=\frac{1}{N-1-i}-\frac{1}{N-1}\le \frac{1}{N},$$
for some $i$. As a result, we have $Ni\le N^2-2N+1+i(1-N).$ Consequently, $i\le\left\lceil\frac{N^2-2N+1}{2N-1}\right\rceil$ that is $i\le \left\lceil \frac{N}{2}\left(1+o\left(\frac{1}{N}\right)\right)\right\rceil$. Thus, we obtain the following neighbourhood for large $N$:
\begin{equation}\label{Def: Nei N-1}
\mathcal{N}_{\rho_1}(N-1) =\left\{ \left\lceil\frac{N}{2}\right\rceil,\ldots,N-1  \right\},\end{equation}
which is the largest finite $\delta$-neighbourhood. Next, the neighbourhood before $\mathcal{N}_{\rho_1}\left(N-1\right)$ is $\mathcal{N}_{\rho_1}\left(\left\lceil\frac{N}{2}\right\rceil-1\right)$ in which the elements satisfy
$${\rho_1}\left(\frac{N}{2}-1-i,\frac{N}{2}-1\right)=\frac{1}{\frac{N}{2}-1-i}-\frac{1}{\frac{N}{2}-1}\le\frac{1}{N},$$
we obtain $i\le\left\lceil\frac{N^2-4N-4}{6N-4}\right\rceil$, and in fact we have $i\le \left\lceil \frac{N}{6}\left(1+o\left(\frac{1}{N}\right)\right)\right\rceil$. Hence, for large $N$

$$\mathcal{N}_{\rho_1}\left(\left\lceil\frac{N}{2}\right\rceil-1\right)=\left\{ \left\lceil\frac{N}{3}\right\rceil,\ldots,\left\lceil\frac{N}{2}\right\rceil-1  \right\}.$$ 
Similarly, for large $N$, these neighbourhoods can be represented in the following form:

\begin{equation}\label{Def: Nei N-j}
\mathcal{N}_{\rho_1}\left(\left\lceil\frac{N}{j}\right\rceil-1\right)=\left\{ \left\lceil\frac{N}{j+1}\right\rceil,\ldots,\left\lceil\frac{N}{j}\right\rceil-1  \right\},\end{equation}

with $\#\mathcal{N}_{\rho_1}\left(\left\lceil\frac{N}{j}\right\rceil-1\right)=\left\lceil\frac{N}{j(j+1)}\left(1+o\left(\frac1{N}\right)\right)\right\rceil$.  To determine $j$, representing the number of such neighbourhoods based on $\delta$, we apply the following condition:

$$ \left\lceil\frac{N}{j+1}\right\rceil-1=N_s\left(=\left\lfloor\frac{1+ \sqrt{1+\frac{4}{\delta}}}{2}\right\rfloor\right),$$
because $\left\lceil\frac{N}{j+1}\right\rceil$ is the smallest natural number among the multi-element neighbourhoods, so $\left\lceil\frac{N}{j+1}\right\rceil-1$ should be equal to $N_s$ which is the largest natural number that has a singleton neighbourhood.
It follows that $j=\left\lceil\frac{2N}{1+\sqrt{1+4N}+2}-1\right\rceil$, where we can rewrite it as $$j=\left\lceil\sqrt{N}\left(1+o\left(\frac{1}{\sqrt{N}}\right)\right)\right\rceil.$$
Now we also add all such neighbourhoods in form of $\mathcal{N}_{\rho_1}\left(\left\lceil\frac{N}{i}\right\rceil-1\right)$ for $1\le i\le \left\lceil\sqrt{N}\right\rceil$ to $\mathcal{V}_\delta$. Hence, $\mathcal{V}_\delta$ is the optimal finite open cover for $\N$ with respect to ${\rho_1}$.

To find the number of open neighbourhoods in $\mathcal{V}_\delta$, for large $N$ we have 

\begin{align*}
K_\delta:=1+ j+N_s&=1+\left\lceil \sqrt{N}+\frac{1+ \sqrt{1+4N}}{2}\left(1+o\left(\frac{1}{\sqrt{N}}\right)\right)\right\rceil\\&=1+\left\lceil \frac{2\sqrt{N}+\sqrt{1+4N}+1}{2}\left(1+o\left(\frac{1}{\sqrt{N}}\right)\right)\right\rceil \\&=1+\left\lceil 2\sqrt{N}\left(1+o\left(\frac{1}{\sqrt{N}}\right)\right)\right\rceil
=\left\lceil2\sqrt{\frac{1}{\delta}}\left(1+o\left(\sqrt{\delta}\right)\right)\right\rceil.
\end{align*}
So, for $\delta>0$, there exists the finite set $$W_\delta:=\left\{1,\ldots, \left\lceil\sqrt{N}\right\rceil-1,\left\lceil\sqrt{N}\right\rceil,\ldots,N-1,N\right\},$$ with $K_\delta$ elements, such that $\N\subset\bigcup_{i=1}^{K_\delta}\mathcal{N}_{\rho_1}(v_i)$, where $v_i\in W_\delta$. We set $V_i:=\mathcal{N}_{\rho_1}(v_i)$ for $1\le i\le K_\delta $, then \begin{equation}\label{Def:V_delta}
    \mathcal{V}_\delta=\{V_i\}_{i=1}^{K_\delta}
\end{equation} is the open cover of $\N$, with respect to which ${\rho_1}$ is totally bounded.

We can now construct, using the collection $\mathcal{V}_\delta=\{V_i\}_{i=1}^{K_\delta}$, a finite collection of cylinders for $\Sigma$. Let $k\ge 1$ and suppose $1\le i_0,\ldots,i_{k-1}\le K_\delta$, then we consider the following sets.

\begin{equation}\label{Def: Z_i_0}
(i_0,\ldots,i_{k-1})_{\mathcal{V}_\delta}:=\left\{ (x_0,\ldots,x_{k-1},\ldots): x_0\in V_{i_0},\ldots,x_{k-1}\in V_{i_{k-1}}  \right\}.\end{equation}

In fact, these sets are a union of cylinders, i.e., $(i_0,\ldots,i_{k-1})_{\mathcal{V}_\delta}=\bigcup_{\i\in \mathrm{U^*}}[\i]$, where $$\mathrm{U^*}=\left\{ (x_0,\ldots,x_{k-1}):  x_0\in V_{i_0},\ldots,x_{k-1}\in V_{i_{k-1}}    \right\}.$$ We are interested in finding $\text{diam}\left((i_0,\ldots,i_{k-1})_{\mathcal{V}_\delta}\right)$. To this end, note that the diameter is determined by elements of different cylinders, since the metric between two points within a single $k$-cylinder from this collection is zero in the first $k$ components, whereas this metric between two members from distinct cylinders is nonzero. These two elements are of the following form:

$$\ulx=(l_{i_0},\ldots,l_{i_{k-1}},\!\!\!\!\!\!\!\!\!\!\!\!\stackrel{\scriptstyle k\text{-th component}}{\stackrel{\small{\downarrow}}1}\!\!\!\!\!\!\!\!\!\!\!\!,1,\ldots),$$
$$\uly=(s_{i_0},\ldots,s_{i_{k-1}},M,M,\ldots),$$
where $M$ is a large number and for all $i_0\le i\le i_{k-1}$ if $V_i\neq\mathcal{N}_{\rho_1}(N)$, $l_i,s_i\in V_i$ are \st $\text{diam}(V_i)={\rho_1}(l_i,s_i)$, and if $V_i=\mathcal{N}_{\rho_1}(N)$, $s_i=N$ and $l_i=l$ \st $l\to\infty$. Then, $$ d_1(\ulx,\underline{y})=\sum_{n=0}^{k-1} \theta^n \text{diam}(V_i)+\sum_{n\ge k} \theta^n \left| 1 - \frac{1}{M} \right|= \delta\sum_{n=0}^{k-1} \theta^n  +\sum_{n\ge k} \theta^n \left| 1 - \frac{1}{M} \right|,$$
where, as $M\to\infty$, we have:
$$ d_1(\ulx,\underline{y})=\delta \sum_{n=0}^{k-1} \theta^n +\sum_{n\ge k} \theta^n = \delta\left(\frac{1-\theta^k}{1-\theta}\right)+\frac{\theta^k}{1-\theta}.$$
Hence, the result is 
\begin{equation} \label{Diam_Union_Full}
\text{diam}\left((i_0,\ldots,i_{k-1})_{\mathcal{V}_\delta}\right)=\delta\left(\frac{1-\theta^k}{1-\theta}\right)+\frac{\theta^k}{1-\theta}.
\end{equation} 

We then define $\U_{\delta}:=\left\{(i_0,\ldots,i_{k-1})_{\mathcal{V}_\delta}: 1\le i_l\le K_{\delta}, \ \forall \ 0\le l\le k-1 \right\}$, which is a finite collection with $\#(\U_{\delta})=(K_\delta)^k$ that covers $\Sigma$.

\textit{Step 2.} We now prove that the cover obtained has properties (1) and (2). In the first step for each $k\ge1$, we obtain a finite collection $\U_{\delta}=\left\{(i_0,\ldots,i_{k-1})_{\mathcal{V}_\delta}: 1\le i_l\le K_{\delta}, \ \forall \ 0\le l\le k-1 \right\}$. For this $\U_{\delta}$ to satisfy properties (1) and (2), we assume that $k=\left\lceil \frac{\log(\delta(1-\theta))}{\log(\theta)}\right\rceil$. Note that $\U_{\delta}$ is a pairwise disjoint cover for $\Sigma$.

We can rewrite $(i_0,\ldots,i_{k-1})_{\mathcal{V}_\delta}$ as an element of $\U_{\delta}$ in the form $(i_0,\ldots,i_{k-1})_{\mathcal{V}_\delta}=\bigcup_{\i\in \mathrm{U^*}}[\textbf{i}]$ where $\mathrm{U^*}=\left\{ (x_0,\ldots,x_{k-1}):  x_0\in V_{i_0},\ldots,x_{k-1}\in V_{i_{k-1}}    \right\}\subset \Sigma^*$ is a countable set such that $|\textbf{i}|= \left\lceil\frac{\log(\delta(1-\theta))}{\log(\theta)}\right\rceil$, which implies property (1).

To obtain property (2), let $\ulx\in\Sigma$. Since $\U_{\delta}$ covers $\Sigma$, then there exists $(i_0,\ldots,i_{k-1})_{\mathcal{V}_\delta}$ \st $\ulx\in(i_0,\ldots,i_{k-1})_{\mathcal{V}_\delta}$. Since $k=\left\lceil \frac{\log(\delta(1-\theta))}{\log(\theta)}\right\rceil$, by what we calculated in \eqref{Diam_Cylinder_Full}, we have $\text{diam}(B_\delta(\ulx))\ge\frac{\theta^k}{1-\theta}$. Now, according to \eqref{Diam_Union_Full}, we have \begin{align*}
\text{diam}\left((i_0,\ldots,i_{k-1})_{\mathcal{V}_\delta}\right)&=\delta\left(\frac{1-\theta^{\left\lceil \frac{\log(\delta(1-\theta))}{\log(\theta)}\right\rceil}}{1-\theta}\right)+\frac{\theta^{\left\lceil \frac{\log(\delta(1-\theta))}{\log(\theta)}\right\rceil}}{1-\theta}\le\delta\left(\frac{1-\theta^{\left\lceil \frac{\log(\delta(1-\theta))}{\log(\theta)}\right\rceil}}{1-\theta}\right)+\delta\\&= \delta \left( \frac{1-\theta^{\left\lceil \frac{\log(\delta(1-\theta))}{\log(\theta)}\right\rceil}}{1-\theta}+1\right).
\end{align*} Since $1\le\frac{1-\theta^k}{1-\theta}<\frac{1}{1-\theta}$ for all $k>0$, we obtain $$\text{diam}(B_{\delta}(\ulx))=\delta\le \delta\left(\frac{1-\theta^{\left\lceil \frac{\log(\delta(1-\theta))}{\log(\theta)}\right\rceil}}{1-\theta}\right)< \text{diam}\left((i_0,\ldots,i_{k-1})_{\mathcal{V}_\delta}\right),$$ and consequently $B_{\delta}(\ulx)\subset (i_0,\ldots,i_{k-1})_{\mathcal{V}_\delta}.$
Setting $1<T:= \frac{1}{1-\theta}+1<\infty$, we have 
$$\text{diam}(B_{T\delta}(\ulx))=\delta \left(\frac{1}{1-\theta}+1\right)> \delta\left(\frac{1-\theta^{\left\lceil \frac{\log(\delta(1-\theta))}{\log(\theta)}\right\rceil}}{1-\theta}+1\right)\ge \text{diam}\left((i_0,\ldots,i_{k-1})_{\mathcal{V}_\delta}\right),$$

we conclude that
$(i_0,\ldots,i_{k-1})_{\mathcal{V}_\delta} \subset B_{T\delta}(\ulx).$

Hence, for $k=\left\lceil \frac{\log(\delta(1-\theta))}{\log(\theta)}\right\rceil$, we found $T=\frac{1}{1-\theta}+1$, such that for all $\ulx\in \Sigma$, there exists $(i_0,\ldots,i_{k-1})_{\mathcal{V}_\delta}\in\U_{\delta}$ \st $\ulx\in(i_0,\ldots,i_{k-1})_{\mathcal{V}_\delta}\in\U_{\delta}$ and we have  $B_{\delta}(\ulx)\subset(i_0,\ldots,i_{k-1})_{\mathcal{V}_\delta}\subset B_{T\delta}(\ulx)$.

Notice that, because $k=\left\lceil \frac{\log(\delta(1-\theta))}{\log(\theta)}\right\rceil$ and $K_\delta=\left\lceil2\sqrt{\frac{1}{\delta}}\left(1+o\left(\sqrt{\delta}\right)\right)\right\rceil$, the total number of elements in $\U_{\delta}$ is consequently given by $$\#(\U_{\delta})=(K_\delta)^k=\left\lceil\left(2\sqrt{\frac{1}{\delta}}\left(1+o\left(\sqrt{\delta}\right)\right)\right) ^ { \frac{\log(\delta(1-\theta))}{\log(\theta)}}\right\rceil.$$

Here, it should be noted that $\U_{\delta}$ is minimal. In fact, for any collection of cylinders of length less than $k$ in $\Sigma$, this collection can not be obtained. So, the desired conclusion of the proposition is obtained, and the proof is complete.
\end{proof}

\begin{remark}
Note that since $\U_{\delta}$ is a partition, property (2) for $\U_\delta$ is equivalent to (U)(a) in condition \hyperref[Condition (U)]{\textbf{(U)}}. In addition, property (1) is the same as (U)(d) in condition \hyperref[Condition (U)]{\textbf{(U)}}. (U)(b) and (U)(c) in condition \hyperref[Condition (U)]{\textbf{(U)}} are also the natural properties for $\U_{\delta}$.
\end{remark}

Note that at the first step of proving Proposition \ref{Lemma:Totally bounded in Full shift}, we actually showed how to use the totally boundedness of $\rho_1$ for $\N$ to show that $d_1$ is totally bounded for $\Sigma$.

There exists a connection between the expected hitting time (\ref{Def: hitting time}) and the expected cover time, which plays an important role in the proof of our results. To this end, we first present a remark that we use in the proof of the following theorem.



\begin{remark}\label{Remark: thida}
The collection $\U_\delta$ inherits the $\psi$-mixing property. In fact, for each $(i_0,\ldots,i_{k-1})_{\mathcal{V}_\delta}$ and $(z_0,\ldots,z_{k-1})_{\mathcal{V}_\delta}$ in $\U_\delta$ we can consider
$$ \frac{{\mu}\left(\cup\left\{[x_{i_0},\ldots,x_{i_{k-1}},y_0,\ldots,y_{j-1},x_{z_0},\ldots,x_{z_{k-1}}]:x_{i_l}\in V_{i_l},x_{z_l}\in V_{z_l},\ 0\le l\le k-1\right\}\right) }{{\mu}((i_0,\ldots,i_{k-1})_{\mathcal{V}_\delta}){\mu}((z_0,\ldots,z_{k-1})_{\mathcal{V}_\delta})},$$
where the union is taken over $\Sigma_j$, all words $y_0, \ldots, y_{j-1}$ of length $j$, and since $(i_0,\ldots,i_{k-1})_{\mathcal{V}_\delta}$ and $(z_0,\ldots,z_{k-1})_{\mathcal{V}_\delta}$ are disjoint, this can be rewritten as 
\begin{align*}
     &\frac{\sum_{\substack{x_{i_l}\in V_{i_l} \\ 0\le l\le k-1}} \ \ \sum_{{\substack{x_{j_l}\in V_{j_l} \\ 0\le l\le k-1}}}{\mu}(\cup_{y_0, \ldots, y_{j-1}\in \Sigma_j}[x_{i_0},\ldots,x_{i_{k-1}},y_0,\ldots,y_{j-1},x_{z_0},\ldots,x_{z_{k-1}}])}{\sum_{\substack{x_{i_l}\in V_{i_l} \\ 0\le l\le k-1}} \ {\mu}([x_{i_0},\ldots,x_{i_{k-1}}])\sum_{{\substack{x_{j_l}\in V_{j_l} \\ 0\le l\le k-1}}} {\mu}([x_{z_0},\ldots,x_{z_{k-1}}])}\\& \le \frac{\sum_{\substack{x_{i_l}\in V_{i_l} \\ 0\le l\le k-1}} \ \ \sum_{{\substack{x_{j_l}\in V_{j_l} \\ 0\le l\le k-1}}}{\mu}(\cup_{y_0, \ldots, y_{j-1}\in \Sigma_j}[x_{i_0},\ldots,x_{i_{k-1}},y_0,\ldots,y_{j-1},x_{z_0},\ldots,x_{z_{k-1}}])}{\sum_{\substack{x_{i_l}\in V_{i_l} \\ 0\le l\le k-1}} \sum_{{\substack{x_{j_l}\in V_{j_l} \\ 0\le l\le k-1}}}\ {\mu}([x_{i_0},\ldots,x_{i_{k-1}}]) {\mu}([x_{z_0},\ldots,x_{z_{k-1}}])}\\& \le \frac{\sum_{\substack{x_{i_l}\in V_{i_l} \\ 0\le l\le k-1}} \ \ \sum_{{\substack{x_{j_l}\in V_{j_l} \\ 0\le l\le k-1}}}{\mu}([x_{i_0},\ldots,x_{i_{k-1}}]) {\mu}([x_{z_0},\ldots,x_{z_{k-1}}])\left(\psi(j)+1\right)}{\sum_{\substack{x_{i_l}\in V_{i_l} \\ 0\le l\le k-1}} \sum_{{\substack{x_{j_l}\in V_{j_l} \\ 0\le l\le k-1}}}\ {\mu}([x_{i_0},\ldots,x_{i_{k-1}}]) {\mu}([x_{z_0},\ldots,x_{z_{k-1}}])}
     =\psi(j)+1,
\end{align*}
where in the last inequality we applied the upper bound of $\psi$-mixing property for cylinders. Similarly, we can find the lower bound, so we have the $\psi$-mixing property for elements in $\U_\delta$.

\end{remark}
\begin{notation*}
We introduce some notation that helps us for the sake of brevity in the proof of the next theorem. Assume that $U=(i_0,\ldots,i_{k-1})_{\mathcal{V}_\delta}, W=(j_0,\ldots,j_{k-1})_{\mathcal{V}_\delta}\in\U_\delta$ and $n, m\in\N_0$ \st $n\ge m$. If we are interested in points in $U$ whose their $m$-th to $n$-th iterates under the shift map have no intersection with the set $W$, we denote the set of all such points by
$$\F_{m-1}^{n-1}(U:W):=\left\{ \ulx\in U:W\cap  \bigcup_{i=m-1}^{n-1} [x_{i+1},\ldots,x_{i+k}]=\varnothing  \right\},$$ 
so, in the case we seek the same property for points in $\Sigma$, we use $\F_{m-1}^{n-1}(\Sigma:W)$, and the set of all points in $U$, where their all $n$ initial iterates of the shift map do not hit $W$ is $\F_0^{n-1}(U:W)$.
Moreover, we use ${\F_{m-1}^{n-1}(U:W)}^c$ for the complement of $\F_{m-1}^{n-1}(U:W)$, i.e. $${\F_{m-1}^{n-1}(U:W)}^c:=\left\{\ulx\in U:W\cap  \bigcup_{i=m-1}^{n-1} [x_{i+1},\ldots,x_{i+k}]\neq \varnothing  \right\}.$$
\end{notation*}

The following theorem not only contributes to the proof of the main result but is also of independent interest.
In \cite[Theorem 4.1]{MikeNatali}, bounds for $\E(\tau_U)$ were obtained, where $U$ is a subinterval i.e., for some $0<c<C<\infty$,
$$\frac{c}{\mu( U)} \le\E_\mu\left(\tau_{U}\right)\le \frac{C}{\mu( U)}.$$
We now present, in the following theorem, the corresponding result for $\E_\mu\left(\tau_{U}\right)$ depending on metric $d_1$, where $U\in\U_\delta$. The key point is that the bounds for this result involve uniform constants as in the interval case.

\begin{theorem}\label{Lemma: upper bound for E(hitting time)}
Assume that $(\Sigma,\sigma)$ is the countable full shift equipped with the metric $d_1$ as in \eqref{Def:Metric d}. Let $\phi:\Sigma\to\R$ be a potential \st $\phi\big|_{[n]}\le -\kappa\log (n+1)$ for some $\kappa>1$. Suppose that $ U=(i_0,\ldots,i_{k-1})_{\mathcal{V}_\delta}$ is an element of finite open cover $\mathcal{U}_\delta$ as defined in Proposition \ref{Lemma:Totally bounded in Full shift}. Then there exists $0<c<1<C<\infty$ such that 
\begin{equation}\label{Equ:Prop of bounds for E(tau)}
\frac{c}{\mu( U)} \le\E_\mu\left(\tau_{U}(\ulx)\right)\le \frac{C}{\mu( U)}.\end{equation}
\end{theorem}
\begin{proof}
 The proof of this theorem relies on the use of Kac's lemma \cite[Lemma 2$'$]{Kac47}. This lemma states that for every measurable set $A$, the mean of $\mu_A \left( \left\{\ulx\in A: \tau_A(\ulx)\ge n \right\}\right)$ in an ergodic system is equal to $\frac{1}{\mu(A)}$. Therefore, the approach taken in this proof is to establish an inequality between $\mu_U \left( \left\{\ulx\in U: \tau_U(\ulx)\ge n \right\}\right)$ and $\mu \left( \left\{\ulx\in U: \tau_U(\ulx)\ge n \right\}\right)$, which was obtained earlier in \cite[Proposition 3.1]{MikeNatali} for cylinders. Then, by summing both sides, we obtain their respective mean. As a result, one side of the inequality corresponds to $\E_\mu(\tau_U(\ulx))$, while the other side is expressed in terms of $\frac{1}{\mu(U)}$, which serves as a bound. We first present some preliminaries used in both bounds, then establish the upper bound, and finally derive the lower bound on \eqref{Equ:Prop of bounds for E(tau)}.

Recall that $k=\left\lceil \frac{\log(\delta(1-\theta))}{\log(\theta)}\right\rceil$ for some $\delta>0$ and $0<\theta<1$. Fix $m\in \N $ and assume that $n\ge k+m$. We have the following expression
\begin{align*}
&\sigma^{-1}\left(\left\{\ulx\in U: \tau_U(\ulx)\ge n\right\}\right)\\&\quad\qquad\quad \qquad\qquad= \sigma^{-1}\left(\left\{ \ulx\in U: \ulu\notin \bigcup_{i=0}^{n-1} [x_{i+1},\ldots,x_{i+k}],\ \forall \ulu\in U \right\}\right) \\&\quad\qquad\quad \qquad\qquad= \left\{ \ulx\in U: \ulu\notin \bigcup_{i=0}^{n-2} [x_{i+1},\ldots,x_{i+k}],\ \forall \ulu\in U \right\} = \F_0^{n-2}(U:U).\end{align*} 
So, since the system is measure preserving then,
\begin{equation} 
\label{Eq: Same measure for sigma inverse}
\mu_U \left(\left\{\ulx\in U: \tau_U(\ulx)\ge n \right\}\right)=\mu_U\left(\F_0^{n-2}(U:U)\right). 
\end{equation}
To use the $\psi$-mixing property, given $m\in\N$, we can rewrite $\mu_U\left(\F_0^{n-2} (U:U)\right)$ as 
\begin{equation}
\mu_U\left(\F_{k+m}^{n-2}(U:U)\right)  -\mu_U\left( {\F_0^{k+m-1}(U:U)}^c\cap \F_{k+m}^{n-2}(U:U) \right).\label{Terms I&II}\end{equation}
We can consider $\F_{k+m}^{n-2}(U:U)$ as follows
\begin{align*}
    &\F_{k+m}^{n-2}(U:U)\\&\ \ = \bigcup\left\{[u_0,\ldots,u_{k-1},x_k,\ldots,x_{k+m},x_{k+m+1},\ldots,x_{n+k-2}]:\ulu\notin \bigcup_{i=k+m}^{n-2} [x_{i+1},\ldots,x_{i+k}], \forall \ulu\in U\right\},
\end{align*}
where the union is taken over all words $(x_k,\ldots,x_{k+m})\in\Sigma_m$ of length $m$ such that $$(u_0,\ldots,u_{k-1},x_k,\ldots,x_{k+m},x_{k+m+1},\ldots,x_{n+k-2})\in\Sigma_{n+k-2}.$$ 

The elements in ${\F_{k+m}^{n-2}(U:U)} $ have a nice property: on each cylinder $$[u_0,\ldots,u_{k-1},x_k,\ldots,x_{k+m},x_{k+m+1},\ldots,x_{n+k-2}],$$ a part $(x_k,\ldots,x_{k+m})$ which can be considered as a gap (with length $m$) between two subsets, $U$ at the beginning and at the end of the cylinder, all $(x_{k+m+1},\ldots,x_{n+k-2})\in \Sigma_{n-m-3}$ where $\ulu\notin \bigcup_{i=k+m}^{n-2} [x_{i+1},\ldots,x_{i+k}],\ \forall \ u_0\in V_{i_0},\ldots,u_{k-1}\in V_{i_{k-1}}$, which has equal measure to $\F_{0}^{n-k-m-2}(\Sigma:U)$. Hence, we can apply the property of $\psi$-mixing \eqref{Equ:psi mixing} on ${\F_{k+m}^{n-2}(U:U)} $ by Remark \ref{Remark: thida},

\begin{equation}\label{Inequality: psi-mixing of U}
    1-\psi(m)\le\frac{\mu\left( {\F_{k+m}^{n-2}(U:U)} \right)}{\mu(U)\mu\left( \F_{0}^{n-k-m-2}(\Sigma:U) \right)} \le \psi(m)+1.\end{equation}
    \textit{Proof of upper bound on \eqref{Equ:Prop of bounds for E(tau)}.} We divide the proof into two steps. First, we establish a lower bound for \eqref{Eq: Same measure for sigma inverse}, and in the second part, we obtain an upper bound for \eqref{Equ:Prop of bounds for E(tau)}. 
    
    \textbf{First step.} To obtain a lower bound on \eqref{Eq: Same measure for sigma inverse}, using \eqref{Terms I&II} we derive a lower bound for $\mu_U\left(\F_0^{n-2} (U:U)\right)$, which in turn provides a lower bound for $\mu_U \left(\left\{\ulx\in U: \tau_U(\ulx)\right\}\right)$ by \eqref{Eq: Same measure for sigma inverse}. To find a lower bound for $\mu_U\left(\F_0^{n-2} (U:U)\right)$, we first find a lower bound for $\mu_U\left(\F_{k+m}^{n-2}(U:U)\right)$ in terms of $k$ and $m$ and then an upper bound for $\mu_U\left( {\F_0^{k+m-1}(U:U)^c}\cap \F_{k+m}^{n-2}(U:U) \right)$. We can find a lower bound for $\mu_U\left(\F_{k+m}^{n-2}(U:U)\right)$ by \eqref{Inequality: psi-mixing of U} as below,
\begin{align*}
     &\mu_U\left(\F_{k+m}^{n-2}(U:U)\right)=\frac{1}{\mu(U)} \mu\left(\F_{k+m}^{n-2}(U:U)\right)\ge \frac{1}{\mu(U)} \mu(U)\mu\left( \F_{0}^{n-k-m-2}(\Sigma:U) \right)(1-\psi(m))\\&= \mu \left(\left\{\ulx\in \Sigma: \tau_U(\ulx)\ge n-m-k \right\}\right)(1-\psi(m)).
\end{align*}
 Now, to find an upper bound for $\mu_U\left( {\F_0^{k+m-1}(U:U)}^c\cap \F_{k+m}^{n-2}(U:U) \right)$, we are dealing with three almost independent conditions (due to the $\psi$-mixing property) on $U$. In addition to requiring the elements of the set to belong to $U$, there are two additional conditions. 

To find an upper bound for $\mu_U\left( {\F_0^{k+m-1}(U:U)^c}\right)$, we rewrite this
\begin{align*}
&\mu_U\left( {\F_0^{k+m-1}(U:U)^c}\right)\\& \ \ \ \ =\mu_U\left(\left\{ \ulx\in U: \ \exists\ u_0\in V_{i_0},\ldots,u_{k-1}\in V_{i_{k-1}},\ [u_0,\ldots,u_{k-1}]\in \left\{[x_{i+1},\ldots,x_{i+k}]\right\}_{i=0}^{k+m-1} \right\}\right).\end{align*}

As $U=\bigcup_{\mathbf{i}\in \mathrm{U^*}}[\mathbf{i}]$ by property (1) in Proposition \ref{Lemma:Totally bounded in Full shift} for $\U_\delta$, where this union is pairwise disjoint, we have

\begin{align*}
   & \mu_U\left(\left\{ \ulx\in U:\ \exists\ u_0\in V_{i_0},\ldots,u_{k-1}\in V_{i_{k-1}}, \ [u_0,\ldots,u_{k-1}]\in \left\{[x_{i+1},\ldots,x_{i+k}]\right\}_{i=0}^{k+m-1} \right\}\right)\\&\ \ \ \ =\frac{ \mu\left(\left\{ \ulx\in U:\ \exists\ u_0\in V_{i_0},\ldots,u_{k-1}\in V_{i_{k-1}}, \ [u_0,\ldots,u_{k-1}]\in \left\{[x_{i+1},\ldots,x_{i+k}]\right\}_{i=0}^{k+m-1} \right\}\right)}{\mu(U)}\\& \ \ \ \  = \frac{\mu\left(\bigcup_{\mathbf{i}\in \mathrm{U^*}}\left\{ \ulx\in [\i]:\ [\i]\in \left\{[x_{i+1},\ldots,x_{i+k}]\right\}_{i=0}^{k+m-1} \right\}\right)}{\sum_{\mathbf{i}\in \mathrm{U^*}}\mu([\textbf{i}]) }\\& \ \ \ \ \le \frac{\sum_{\mathbf{i}\in \mathrm{U^*}}\mu\left(\left\{ \ulx\in [\i]:\ [\i]\in \left\{[x_{i+1},\ldots,x_{i+k}]\right\}_{i=0}^{k+m-1} \right\}\right)}{\sum_{\mathbf{i}\in \mathrm{U^*}}\mu([\textbf{i}]) }\\&
   \ \ \ \ = \frac{\sum_{\mathbf{i}\in \mathrm{U^*}}\mu_{[\textbf{i}]}\left(\left\{ \ulx\in [\i]:\ [\i]\in \left\{[x_{i+1},\ldots,x_{i+k}]\right\}_{i=0}^{k+m-1} \right\}\right) \mu([\textbf{i}])}{\sum_{\mathbf{i}\in \mathrm{U^*}}\mu([\textbf{i}]) } \le \tilde{\theta},
\end{align*}

where in the last inequality we use a result in \cite[p16]{MikeNatali}, where they proved for each $\i$ there exists a uniform $\tilde{\theta}<1$ such that
$ \mu_{[\i]}\left( \left\{ \ulx\in [\i]: [\i]\in \left\{[x_{i+1},\ldots,x_{i+k}]\right\}_{i=0}^{k+m-1} \right\}\right)\le \tilde{\theta}.$ Hence, we conclude that $\mu_U\left( {\F_0^{k+m-1}(U:U)^c}\right) \le \tilde{\theta}$. 


 
To find an upper bound for $\mu_U(\F_{k+m}^{n-2}(U:U))$, we use the upper bound in \eqref{Inequality: psi-mixing of U}, hence:

\begin{align*}
    &\mu_U\left( {\F_0^{k+m-1}(U:U)^c}\cap \F_{k+m}^{n-2}(U:U) \right)  = \mu_U\left( {\F_0^{k+m-1}(U:U)^c}\right)\mu_U\left(\F_{k+m}^{n-2}(U:U)\right)(1+\psi(m))\\&\qquad \qquad \qquad\qquad\ \le \left(\tilde{\theta}\cdot \frac{1}{\mu(U)} \cdot\mu(U) \mu \left(\left\{\ulx\in \Sigma: \tau_U(\ulx)\ge n-m-k \right\}\right)\left( \psi(m)+1\right)\right)(1+\psi(m))\\&\qquad\qquad\qquad\qquad  \ = \left(\tilde{\theta} \mu \left(\left\{\ulx\in \Sigma: \tau_U(\ulx)\ge n-m-k \right\}\right) \left( \psi(m)+1\right)^2\right),
\end{align*}
where in the first equality, we use Remark \ref{Remark: thida} with the error term $(1+\psi(m))$. Hence, we can find the upper bound for $\E(\tau_U(\ulx))$ as follows, by \eqref{Terms I&II} and the above estimates,
\begin{align*}
    &\mu_U \left( \left\{\ulx\in U: \tau_U(\ulx)\ge n \right\}\right) \ge \bigg(\mu \left\{\ulx\in \Sigma: \tau_U(\ulx)\ge n-m-k \right\}(1-\psi(m)) \\& \qquad \quad \qquad \qquad \qquad \qquad \qquad \qquad \ \ \ \ \ - \tilde{\theta}\left( \psi(m)+1\right)^2 \mu \left(\left\{\ulx\in \Sigma: \tau_U(\ulx)\ge n-m-k \right\}\right)\bigg)\\& \ \ \qquad \quad \qquad \qquad \ \ \ \ \ \ \ \ \ \ = \left((1-\psi(m))- \tilde{\theta}\left( \psi(m)+1\right)^2\right) \mu \left(\left\{\ulx\in \Sigma: \tau_U(\ulx)\ge n-m-k \right\}\right).
\end{align*}
As $n\ge k+m$, we always have $\mu \left( \left\{\ulx\in \Sigma: \tau_U(\ulx)\ge n-m-k \right\} \right)\ge  \mu \left(\left\{\ulx\in \Sigma: \tau_U(\ulx)\ge n \right\} \right)$,  then,
$$\mu_U \left(\left\{\ulx\in U: \tau_U(\ulx)\ge n \right\}\right)\ge   \left((1-\psi(m))- \tilde{\theta}\left( \psi(m)+1\right)^2\right)   \mu \left(\left\{\ulx\in \Sigma: \tau_U(\ulx)\ge n \right\} \right).$$

We can fix some large enough $m$ such that $\psi(m)$ is small enough, then
\begin{equation}\label{Ineq: lwr bound for 1st step of upr bound for E(tau)}
\mu_U \left(\left\{\ulx\in U: \tau_U(\ulx)\ge n \right\}\right)\ge ({1-\tilde{\theta}}) \mu \left(\left\{\ulx\in \Sigma: \tau_U(\ulx)\ge n \right\} \right).\end{equation}

\textbf{Second step.} We now use the inequality \eqref{Ineq: lwr bound for 1st step of upr bound for E(tau)} obtained in the first step to derive an upper bound for \eqref{Equ:Prop of bounds for E(tau)}. By summing over $n$ in \eqref{Ineq: lwr bound for 1st step of upr bound for E(tau)}, for large $n$,

$$\frac{\sum_n \mu_U \left( \left\{\ulx\in U: \tau_U(\ulx)\ge n \right\}\right)}{1-\tilde{\theta}}\ge  \sum_n \mu \left( \left\{\ulx\in \Sigma: \tau_U(\ulx)\ge n \right\}\right)= \E(\tau_U(\ulx)).$$

Now, by Kac's lemma we have $\sum_n \mu_U \left( \left\{\ulx\in U: \tau_U(\ulx)\ge n \right\}\right)=\frac{1}{\mu(U)}$. Thus, we obtain
$$ \frac{1}{(1-\tilde{\theta})}\cdot\frac{1}{\mu(U)}\ge \E(\tau_U(\ulx)).$$

Since $\tilde{\theta}<1$, setting $C:=\frac{1}{1-\tilde{\theta}}>1$, $$\frac{C}{\mu(U)}\ge \E(\tau_U(\ulx)).$$
\textit{ Proof of lower bound on \eqref{Equ:Prop of bounds for E(tau)}}. To find a lower bound for $ \E(\tau_U(\ulx))$, we use \eqref{Terms I&II} and the lower bound of \eqref{Inequality: psi-mixing of U} to obtain
\begin{align}\label{Ineq: used in proof of lower bound of E}
   \notag \mu_U \left( \left\{\ulx\in U: \tau_U(\ulx)\ge n \right\}\right)& \notag=\mu_U\left(\F_{k+m}^{n-2}(U:U)\right)  -\mu_U\left( {\F_0^{k+m-1}(U:U)^c}\cap \F_{k+m}^{n-2}(U:U) \right)\\& \notag \le \mu_U\left(\F_{k+m}^{n-2}(U:U)\right)=\frac{1}{\mu(U)} \mu\left(\F_{k+m}^{n-2}(U:U)\right) \\&\notag \le \frac{1}{\mu(U)} \mu \left(\left\{\ulx\in \Sigma: \tau_U(\ulx)\ge n-m-k \right\}\right) \mu(U) \left( \psi(m)+1\right)\\&=  \mu \left(\left\{\ulx\in \Sigma: \tau_U(\ulx)\ge n-m-k \right\}\right) \left( \psi(m)+1\right).
\end{align}
Now, based on (\cite[Theorem 3.7]{Saeed}), which states that in any symbolic dynamical system with $\psi$-mixing measure, we have the estimate $\mu \left(\left\{\ulx\in \Sigma: \tau_U(\ulx)\ge n \right\}\right)=  e^{-\mu(U)(n)}(1+o(1))$, for large $n$, we conclude that
\begin{align*}
\mu \left(\left\{\ulx\in \Sigma: \tau_U(\ulx)\ge n-m-k \right\}\right)&=  e^{-\mu(U)(n-m-k)}(1+o(1))=e^{-\mu(U)n}e^{\mu(U)(m+k)}(1+o(1))\\& = \mu \left(\left\{\ulx\in \Sigma: \tau_U(\ulx)\ge n \right\}\right) e^{\mu(U)(m+k)} (1+o(1)).\end{align*}
Using this approximation, we can rewrite $\eqref{Ineq: used in proof of lower bound of E}$ as follows.
$$\frac{\mu_U \left( \left\{\ulx\in U: \tau_U(\ulx)\ge n \right\}\right)}{ \left( \psi(m)+1\right)}\le  \mu \left(\left\{\ulx\in \Sigma: \tau_U(\ulx)\ge n \right\}\right) e^{\mu(U)(m+k)}(1+o(1)).$$
Taking the summation over $n$ and using Kac's lemma, we conclude,
\begin{equation}\label{inequality: proof of lower bound for E_mu}
\frac{1}{ \mu(U)\left( \psi(m)+1\right)}\le e^{\mu(U)(m+k)} (\E_\mu(\tau_U(\ulx))+C'),\end{equation}
where $\sum_n \mu \left( \left\{\ulx\in \Sigma: \tau_U(\ulx)\ge n \right\}\right)o(1)=C'$. It remains to bound $\mu(U)(m+k)$. We first estimate $\mu(U)$. By Definition \ref{Definition:Gibbs measure in CMS} of the Gibbs measure there exists $K$ such that for all $u_0\in V_{i_0},\ldots,u_{k-1}\in V_{i_{k-1}}$ we have $\mu([u_0,\ldots,u_{k-1}])\le K e^{S_k\phi(\underline{x})}$ for all $\underline{x}\in [u_0,\ldots,u_{k-1}]$. As we assume that $\phi\big|_{[n]}\le -\kappa\log (n+1)<0$ for some $\kappa>1$, then for all $\underline{x}\in [u_0,\ldots,u_{k-1}]$, 
\begin{equation*}
    \mu([u_0,\ldots,u_{k-1}])\le K e^{S_k{\phi}(\underline{x})}\le K \prod_{j=0}^{k-1} e^{\log (u_j+1)^{-\kappa}}= K \prod_{j=0}^{k-1}(u_j+1)^{-\kappa}.
\end{equation*}
Now, to calculate $\mu(U)$, there are three general cases for $U\in\U_\delta$ based on the number of elements in $V_{i_j}$. Since one of the $V_{i_j}$ is infinite, i.e. $\mathcal{N}_{\rho_1}(N)$ which is defined in \eqref{Countable set of N}, first case can be considered as all $V_{i_j}$ for $U$ are finite, the second case is that all $V_{i_j}$ are $\mathcal{N}_{\rho_1}(N)$ and the last case is where $k'$ of $V_{i_j}$ are finite for $1<k'<k$ but the others are $\mathcal{N}_{\rho_1}(N)$. In the following, we derive an upper bound for $\mu(U)$ for each of these cases.

For the first case, we obtain
\begin{align}\label{calculating_mu}
\notag \mu(U)&=\mu\left(\bigcup_{u_j\in V_{i_j},\ 0\le j\le k-1} [u_0,\ldots,u_{k-1}]\right)= \sum_{u_j\in V_{i_j},\ 0\le j\le k-1}\mu([u_0,\ldots,u_{k-1}])\\&\le K \sum_{u_j\in V_{i_j},\ 0\le j\le k-1}\prod_{j=0}^{k-1}(u_j+1)^{-\kappa} \le K  \prod_{j=0}^{k-1}\sum_{u_j\in V_{i_j}} (u_j+1)^{-\kappa} \le K \prod_{j=0}^{k-1} \#V_{i_j} (\min_{n\in V_{i_j}} (n)+1)^{-\kappa}.
\end{align}

In this case, as $V_{i_j}$ is not $\mathcal{N}_{\rho_1}(N)$, we have two general behaviours for the term $\#V_{i_j} (\min_{n\in V_{i_j}} (n)+1)^{-\kappa}$. Either $V_{i_j}$ is a singleton, which implies that
$$\#V_{i_j} (\min_{n\in V_{i_j}} (n)+1)^{-\kappa}=  (\min_{n\in V_{i_j}} (n)+1)^{-\kappa}\le 2^{-\kappa},$$
or $1<\#V_{i_j}<\infty$, in which case, according to \eqref{Def: Nei N-j}, each $V_{i_j}$ has the form of  $\mathcal{N}_{\rho_1}\left(\left\lceil\frac{N}{n_j}\right\rceil-1\right)$ for some $1\le n_j\le \sqrt{N}$, \st $\#V_{i_j}=\#\mathcal{N}_{\rho_1}\left(\left\lceil\frac{N}{n_j}\right\rceil-1\right)=\left\lceil\frac{N}{n_j^2}\left(1+o\left(\frac1{N}\right)\right)\right\rceil$ and $\min_{n\in V_{i_j}}=\left\lceil\frac{N}{n_j+1}\right\rceil$. Then, 
\begin{align*}
    \#V_{i_j} (\min_{n\in V_{i_j}} (n)+1)^{-\kappa}&= \frac{N}{n_j^2}\left(1+o\left(\frac1{N}\right)\right) \left(\frac{N}{n_j+1}+1\right)^{-\kappa}\\&= \frac{N}{n_j^2}\left(1+o\left(\frac1{N}\right)\right)\left(\frac{N+n_j+1}{n_j+1}\right)^{-\kappa}\\& \le   N\left(1+o\left(\frac1{N}\right)\right)\left(\frac{N+\sqrt{N}+1}{2}\right)^{-\kappa} \\& \le   N\left(1+o\left(\frac1{N}\right)\right)\left(N\right)^{-\kappa}=  \frac{\left(1+o\left(\frac1{N}\right)\right)}{\left(N\right)^{\kappa-1}}<1,
 \end{align*}

where the last inequality is obtained since $\kappa>1$. Hence, it follows that $$0< \#V_{i_j} (\min_{n\in V_{i_j}} (n)+1)^{-\kappa}<1,$$ for all $0\le j\le k-1$. By considering $\eta_1(k):=\prod_{j=0}^{k-1} \#V_{i_j} (\min_{n\in V_{i_j}} (n)+1)^{-\kappa}$, consequently, for the first case we have $\mu(U)\le  K\eta_1(k)$, where $ \eta_1(k)$ decays exponentially to 0 as $k\to \infty$.

 Now, suppose the second case, where for all $0\le j\le k-1$, $V_{i_j}=\mathcal{N}_{\rho_1}(N)$. Using the same method as for \eqref{calculating_mu}, we have
\begin{align*}
     \mu(U)&\le K  \prod_{j=0}^{k-1}\sum_{u_j\in V_{i_j}} (u_j+1)^{-\kappa}=K  \prod_{j=0}^{k-1}\sum_{u_j=N}^\infty (u_j+1)^{-\kappa}=K  \prod_{j=0}^{k-1}\sum_{u_j=N+1}^\infty (u_j)^{-\kappa}.
\end{align*}
Now, since we have
\begin{equation}\label{Ineq: Case Infin. of Proof of lwbd of E(hit)}
    \sum_{u_j=N+1}^\infty (u_j)^{-\kappa}\le \int_N^\infty x^{-\kappa}dx=\left(\frac{N^{-\kappa+1}}{\kappa-1}\right)<1.
\end{equation}
Hence, also in this case we write $\eta_2(k):= \prod_{j=0}^{k-1}\sum_{u_j=N+1}^\infty (u_j)^{-\kappa}$ and $\mu(U)\le  K\eta_2(k)$, which decays exponentially to 0 as $k\to \infty$.



For the last case of $U$, if we consider the number of finite $V_{i_j}$ is $k'$ for some $1<k'<k$, so there are $k-k'$ infinite $V_{i_j}$ such that all are $\mathcal{N}_{\rho_1}(N)$. In this case we use the same method in \eqref{calculating_mu} to obtain the following.

\begin{align*}
    \mu(U)&\le K  \prod_{j=0}^{k-1}\sum_{u_j\in V_{i_j}} (u_j+1)^{-\kappa}= K  \prod_{j=0}^{k'-1}\sum_{u_j\in V_{i_j}} (u_j+1)^{-\kappa} \prod_{j=k'}^{k-1}\sum_{u_j\in V_{i_j}} (u_j+1)^{-\kappa}\\& \le K \prod_{j=0}^{k'-1} \#V_{i_j} (\min_{n\in V_{i_j}} (n)+1)^{-\kappa} \prod_{j=k'}^{k-1}\sum_{u_j=N+1}^\infty (u_j)^{-\kappa}=K \eta_1(k'-1)\eta_2(k-k').
\end{align*}

Since $\eta_1$ and $\eta_2$ decay exponentially to 0, and taking $\eta_3(k):=\eta_1(k).\eta_2(k)$, in the third case also we obtain $\mu(U)\le K\eta_2(k)$, where $\eta_2$ decays exponentially to 0.

This leads us to the conclusion that for all $U\in \U_\delta$, $\mu(U)=\mathcal{O}\left(\eta(k)\right)$, \st $\eta(k)$ has exponential behaviour converging to 0. Thus, we have $e^{\mu(U)(m+k)}\le e^{\eta(k)(m+k)}$.
  
Hence, using this inequality in \eqref{inequality: proof of lower bound for E_mu},

$$\frac{^{1}}{\mu(U)\left( \psi(m)+1\right)}\le \left(e^{\eta(k)(m+k)}\right)\E_\mu(\tau_U(\ulx))+C'.$$

We can fix some large enough $m$ such that $\psi(m)$ is small enough, then for large $k$ we have
$$\frac{c}{\mu(U)}\le\frac{1}{\mu(U)}-C'\le \E_\mu(\tau_U(\ulx)),$$ for some $c<1$. So, the proof of the theorem is complete.
\end{proof}

\begin{remark}
    A statement analogous to Theorem \ref{Lemma: upper bound for E(hitting time)} can be obtained for any cylinder in a countable Markov shift with a Gibbs measure.
\end{remark}
As we use the method used in \cite[Section 5]{MikeNatali} to prove our results, we state a lemma that was established in the proof of Theorem \ref{Th:MN_main_theorem}. Before stating the lemma, we consider the following definitions.

\begin{definition}
Let $\mathcal{U}$ be a finite set of subsets in $\Sigma$. We say $\mathcal{U}$ is a \emph{subpartition} if all elements in this set are pairwise disjoint and there exists a finite or countable set of finite words $\mathrm{U}^*\subset\Sigma^*$ \st $\mathcal{U}=\bigcup_{\mathbf{i}\in \mathrm{U}^*}[\textbf{i}]$.
\end{definition}
Note that a subpartition is not necessarily a partition. Recall that for $\textbf{i}\in \Sigma$, we have $\tau_{\mathcal{U}}(\textbf{i})=\inf\left\{n\ge 1:  \{\textbf{i},\ldots,\sigma^n(\textbf{i})\}  \ \text{hits each}\ U\in \mathcal{U}\right\}$.

We now state the following lemma. To see the proof of this lemma, see \cite[Proof of Theorem 2.2]{MikeNatali}.

\begin{lemma}\label{Lemma:subpartition}
    Assume that $\mathcal{U}=\left\{ \mathrm{U}_1, \ldots,\mathrm{U}_N \right\}$ is a finite subpartition on $\left(\Sigma,\sigma\right)$ equipped with a Gibbs measure $\mu$. Let $\mathcal{U}=\bigcup_{\mathbf{i}\in \mathrm{U}^*}[\mathbf{i}]$ \st $|\mathbf{i}|=L$. Then, there exist $C>1$ \st
    $$  \E_\mu(\tau_\mathcal{U})\le C\left(L+\max_{1\le m\le N}\E_\mu(\tau_{\mathrm{U}_m})\right)\left( \sum_{k=1}^N\frac{1}{k}\right).$$
\end{lemma}

We use this lemma in establishing the upper bounds in our results.

\subsection{Proof of Theorem \ref{Theorem: main theorem}} In this section we present the proof of Theorem \ref{Theorem: main theorem}.
\begin{proof}[ Proof of Theorem \ref{Theorem: main theorem}] We first prove the upper bound and then proceed to prove the lower bound. The idea of the proof for both bounds is to make use of property (2) for $\U_\delta$. Indeed, having a cover time for $\Sigma$ leads to having a hitting time for $\U_\delta$, and vice versa. This is the key point that, with the help of the propositions in Section \ref{Section: Propositions}, completes the proof.

\textit{ Proof of upper bound}. Let $\delta>0$. By Proposition \ref{Lemma:Totally bounded in Full shift}, there exists pairwise disjoint open cover $\U_{\delta}=\left\{(i_0,\ldots,i_{k-1})_{\mathcal{V}_\delta}: 1\le i_l\le K_{\delta}, \ \forall \ 0\le l\le k-1 \right\}$ satisfying properties (1) and (2) for $(\Sigma,\sigma)$, where $K_{\delta}=\left\lceil2\sqrt{\frac{1}{\delta}}\left(1+o\left(\sqrt{\delta}\right)\right)\right\rceil$. In particular, there exist $1<T<\infty$ \st for $\ulx\in\Sigma$, there is a set $(i_0,\ldots,i_{k-1})_{\mathcal{V}_\delta}$ with $B_{\delta}(\ulx)\subset (i_0,\ldots,i_{k-1})_{\mathcal{V}_\delta}\subset B_{T\delta}(\ulx)$. If for all $\ulx\in \Sigma$, $\{\ulx,\ldots,\sigma^n(\ulx)\}$ hits each element of $\U_{\delta/2T}$, then $\{\ulx,\ldots,\sigma^n(\ulx)\}$ is $\delta$-dense in $\Sigma$. Moreover, as $(i_0,\ldots,i_{k-1})_{\mathcal{V}_\delta}\subset B_{2\delta}(\ulx)$ for all $(i_0,\ldots,i_{k-1})_{\mathcal{V}_\delta}\in \U_{\delta/2T}$, we have $\tau_\delta\le \tau_{\U_{\delta/2T}}$. Hence, $$\E_\mu(\tau_\delta)\le \E_\mu(\tau_{\U_{\delta/2T}}).$$ 
For convenience in the computations, we fix $\U_{\delta/2T}=\U_{\delta}$. Since $\U_{\delta}=\bigcup_{\mathbf{i}\in \mathrm{U}^*}[\mathbf{i}]$ is a subpartition for $\Sigma$ \st $|\textbf{i}|= \left\lceil\frac{\log(\delta(1-\theta))}{\log(\theta)}\right\rceil$, where $ \mathrm{U}^*=\left\{(x_0,\ldots,x_{k-1}): \ 0\le l\le k-1, x_l\in V_j, \ 1\le j\le K_\delta\right\}\subset \Sigma^*$ is a countable set, therefore, by Lemma \ref{Lemma:subpartition} there exists constant $C>1$ \st 

  $$ \E_\mu(\tau_\mathcal{U_\delta})\le C\left(\left\lceil\frac{\log(\delta(1-\theta))}{\log(\theta)}\right\rceil+\max_{1\le i_0,\ldots,i_{k-1}\le K_{\delta}}\E_\mu\left(\tau_{(i_0,\ldots,i_{k-1})_{\mathcal{V}_\delta}}\right)\right)\left( \sum_{k=1}^{\#(\U_{\delta})}\frac{1}{k}\right).$$

Now, since $\left( \sum_{k=1}^{\#(\U_{\delta})}\frac{1}{k}\right)\le C' \log \#(\U_{\delta}) $ for some $C'>1$, and in view of Theorem \ref{Lemma: upper bound for E(hitting time)}, where we showed that $\E_\mu\left(\tau_{(i_0,\ldots,i_{k-1})_{\mathcal{V}_\delta}}(\ulx)\right)\le \frac{C}{\mu\left((i_0,\ldots,i_{k-1})_{\mathcal{V}_\delta}\right)}$ for some $C>1$, we conclude 

  \begin{align}
  \notag \E_\mu(\tau_\mathcal{U_\delta})&\le C\left(\left\lceil\frac{\log(\delta(1-\theta))}{\log(\theta)}\right\rceil+\max_{1\le i_0,\ldots,i_{k-1}\le K_{\delta}} \frac{1}{\mu\left((i_0,\ldots,i_{k-1})_{\mathcal{V}_\delta}\right)}\right)\left( \log \#(\U_{\delta}) \right)\\& \label{Ineq: before part of up bound for E_delta}
   \le C\left(\left\lceil\frac{\log(\delta(1-\theta))}{\log(\theta)}\right\rceil+ \frac{1}{ \min_{\ulx\in \Sigma} \mu(B_{\delta/2T}(\ulx))}\right)\left( \log \#(\U_{\delta}) \right)\\&\label{Ineq: part of up bound for E_delta}
    = C\left(\left\lceil\frac{\log(\delta(1-\theta))}{\log(\theta)}\right\rceil+ \frac{1}{ M_\mu(\delta/2T)}\right)\left( \log \#(\U_{\delta}) \right).
\end{align}

Recall that $\#(\U_{\delta})=\left\lceil\left(2\sqrt{\frac{1}{\delta}}\left(1+o\left(\sqrt{\delta}\right)\right)\right) ^ { \frac{\log(\delta(1-\theta))}{\log(\theta)}}\right\rceil$. Now, we have 
$\log \#(\U_{\delta})=C \left(\log \frac{1}{\delta}\right)^2+ o(1).$ Hence, we can replace $\log\#(\U_{\delta})$ by $\left(\log \frac{1}{\delta}\right)^2$ in \eqref{Ineq: part of up bound for E_delta}. Now, since we have $\frac{\log(\delta(1-\theta))}{\log(\theta)}\left(\log \frac{1}{\delta}\right)^2=\mathcal{O}\left(\left(\log \frac{1}{\delta}\right)^3\right)$, we obtain that 
$$ \E_\mu(\tau_\mathcal{U_\delta})\le \left(\frac{C}{ M_\mu(\delta/2T)}\right)\left(\log \frac{1}{\delta}\right)^2.$$

\textit{ Proof of lower bound}. We now find a lower bound for $\E_\mu(\tau_\delta)$. If the set $\{\ulx,\ldots,\sigma^n(\ulx)\}$ is $\delta$-dense in $\Sigma$, then there exists an element of that set in each ball centred at a point in $\Sigma$ of radius $\delta$ or larger, which by property (2) for $\U_{2\delta}$, for all $\ulx\in \Sigma$, there exists $(i_0,\ldots,i_{k-1})_{\mathcal{V}_\delta}\in \U_{2\delta}$ \st $B_{2\delta}(\ulx)\subset (i_0,\ldots,i_{k-1})_{\mathcal{V}_\delta}$, therefore,  $\E_\mu(\tau_{\U_{2\delta}})\le \E_\mu(\tau_\delta)$. 

Now, by the definition, $\E_\mu(\tau_{(i_0,\ldots,i_{k-1})_{\mathcal{V}_\delta}})\le\E_\mu(\tau_{\U_{2\delta}})$, for all $(i_0,\ldots,i_{k-1})_{\mathcal{V}_\delta}\in\U_{2\delta}$. We use the lower bound of Theorem \ref{Lemma: upper bound for E(hitting time)}, where for some $0<c<1$ and for each  $(i_0,\ldots,i_{k-1})_{\mathcal{V}_\delta}\in\U_{2\delta}$ we have
$$ \frac{c}{\mu((i_0,\ldots,i_{k-1})_{\mathcal{V}_\delta})} \le\E_\mu(\tau_{(i_0,\ldots,i_{k-1})_{\mathcal{V}_\delta}}).$$
Since there is some $(i_0,\ldots,i_{k-1})_{\mathcal{V}_\delta}\in \U_{2\delta} $ \st $\mu((i_0,\ldots,i_{k-1})_{\mathcal{V}_\delta})\le \mu (\min_{\ulx\in \Sigma}(B_{2T\delta}(\ulx))=M_\mu({2T\delta})$, therefore, there exists $0<c<1$ \st $$\frac{c}{M_\mu({2T\delta})}\le \E_\mu(\tau_\delta).$$
Therefore, the proof of Theorem \ref{Theorem: main theorem} is completed.
\end{proof}

\section{Cover time under an exponential metric}\label{Section: exponential metric}

In this section, we examine the behaviour of the cover time in $(\Sigma, \sigma)$ equipped with the metric \begin{align*}
    d_2(\ulx,\underline{y})=\sum_{n\ge 0} \theta^n  \left| \frac{1}{e^{\alpha x_n}} - \frac{1}{e^{\alpha y_n}} \right|,
\end{align*}where $\alpha>0$, $\theta\in(0,1)$ and $\ulx,\uly$ are any elements that belong to the full shift. We first derive similar results with respect to metric $d_2$ in the same order as in Section \ref{Section: Propositions}, and then prove Theorem \ref{Theorem: main theorem 2}.
\subsection{Preliminary Propositions for $d_2$} We first find the diameter of the cylinders associated with the metric $d_2(\ulx,\underline{y})=\sum_{n\ge 0} \theta^n  \left| \frac{1}{e^{\alpha x}} - \frac{1}{e^{\alpha y}} \right|,$ where $\alpha>0$ and $\theta\in(0,1)$. 

Considering the cylinder $Z_k$, as described at the beginning of Subsection \ref{subsection: PP for d_1}, we obtain the diameter of the cylinder where $\ulx$ and $\uly$ have the form  
$\ulx=(u_0,\ldots,u_{k-1},M,\ldots)$ and $\uly=(u_0,\ldots,u_{k-1},1,1,\ldots),$ where $M$ is a large number. Then we obtain
$$ d_2(\ulx,\underline{y})=\theta^k\left(\sum_{n\ge 0} \theta^n \left| \frac{1}{e^{\alpha M}} - \frac{1}{e^{\alpha }} \right|\right),$$
and as $M\to \infty$, $d_2(\ulx,\underline{y})=\frac{1}{e^\alpha} \cdot \frac{\theta^k}{1-\theta}$. Hence, the diameter of the cylinder with respect to metric $d_2$ is
\begin{equation}\label{Diam_Cylinder_Full_d_2}
    \text{diam}(Z_k)=\frac{1}{e^\alpha} \cdot \frac{\theta^k}{1-\theta}. 
\end{equation}

\begin{remark}\label{Remark: golden ration}
   We note the behaviour of the two maps that will be used in the proof of the following proposition. The map $h(x)=\frac{1}{x}\log\left(e^{-x}+1\right)$ for $x\in(0,\infty)$ is decreasing, and supposing $ \frac{1}{x}\log\left(e^{-x}+1\right)=1$, we obtain $e^x-e^{-x}-1=0$, which has root $x=\log\varphi$, where $\varphi:=\frac{1+\sqrt{5}}{2}$, the golden ratio. Hence, we observe that $h(x)$ is always less than 1 if $x\ge \log \varphi\approx 0.48$. To simplify, we use the notation $\gamma:=\log \varphi$.
    In addition, the mapping $h'(x)=\frac{1}{x}\log(1-e^{-x})$ is strictly increasing, and for $x\in(0,\gamma]$, $h'(x)<-2$.
\end{remark}

Now, analogously to Proposition \ref{Lemma:Totally bounded in Full shift}, we can construct a finite open cover for $\Sigma$ with respect to $d_2$, which satisfies condition \hyperref[Condition (U)]{\textbf{(U)}}.

\begin{proposition}\label{Lemma:Totally bounded in Full shift d_2}
   
     Let $\delta>0$ and assume that $(\Sigma,\sigma)$ is a countable full shift with the metric $d_2$ as in \eqref{Def:Metric d_2}. Then, for $\alpha>0$, there exists a finite open cover $\U_{\delta}^{\alpha}$ for $(\Sigma,\sigma)$ with $$\#(\U_{\delta}^{\alpha})=\left\lfloor\left( \frac{1}{\alpha} \log \left(\frac{1}{\delta}\right)+\mathcal{O}_\alpha(1)\right) ^ { \log_\theta{\delta e^\alpha(1-\theta)}}\right\rfloor.$$
    Moreover, $\U_{\delta}^{\alpha}$ satisfies the following two conditions
\begin{enumerate}
    \item for all $U\in \U_{\delta}^{\alpha}$, we have $U=\bigcup_{\mathbf{i}\in \mathrm{U^*}}[\mathbf{i}]$ \st $|\mathbf{i}|= \left\lceil\frac{\log(e^{\alpha}\delta(1-\theta))}{\log(\theta)}\right\rceil$, where $\mathrm{U^*}\subset\Sigma^*$ is a set of finite words of length $|\mathbf{i}|$ with respect to $U$,
    \item there exist $1<T<\infty$ \st for all $\ulx\in \Sigma$, there exist  $U\in \U_{\delta}^{\alpha}$ \st $B_{\delta}(\ulx)\subset U \subset B_{T\delta}(\ulx)$.
\end{enumerate}
    
\end{proposition}
\begin{proof}
 The idea of the proof is analogous to that of Proposition \ref{Lemma:Totally bounded in Full shift}, i.e., dividing the proof into two steps. In the first step, we find a finite collection of cylinders that efficiently cover $\Sigma$, and in the second step, we then verify conditions (1) and (2) for this collection. 
 
\textit{Step 1}. Let $\alpha>0$ and assume that we have the metric $d_2(\ulx,\underline{y})=\sum_{n\ge 0} \theta^n  \rho_2(x_n, y_n)$ where $\theta\in(0,1)$ and $\rho_2(x_n, y_n) = \left| \frac{1}{e^{\alpha x_n}} - \frac{1}{e^{\alpha y_n}} \right|$. Here we use the total boundedness of the $\rho_2$, since this allows us to find a finite cover of $\N$ with respect to $\rho_2$. Let $\delta> 0$, and start by letting $\mathcal{V}_\delta^\alpha$ be the empty set, to which elements are then added step by step. Consider the minimal natural number $N$ \st $\frac{1}{e^{\alpha N}}\le \delta$, and as we have $\delta\le\frac{1}{e^{\alpha (N-1)}}$, this implies $\left\lceil\frac{1}{\alpha}\log\frac{1}{\delta}\right\rceil= N$. So, $\frac{1}{e^{\alpha N}}$ is an approximation of the diameter of the $\delta$-neighbourhoods. Moreover,
\begin{align}\label{Countable set of N d_2+}
    \mathcal{N}_{\rho_2}(N):=\{N,N+1,\ldots\},
\end{align}
is the set of all $n\ge N$ \st we have ${\rho_2}(N,n)<\delta$. We add $\mathcal{N}_{\rho_2}(N)$ to $\mathcal{V}_\delta^{\alpha}$.

We now estimate the number of singleton $\delta$-neighbourhoods in $\N$ with respect to the metric $\rho_2$. We use the following formula, where $n-1$ is the smallest element among the non-singleton neighbourhoods, and the preceding number is the largest natural number with the singleton $\delta$-neighbourhood

$$\rho_2(n-1,n)<\frac{1}{e^{\alpha N}}.$$
Now, since 
$$ \frac{1}{e^{\alpha(n-1)}}-\frac{1}{e^{\alpha n}}=\frac{e^{\alpha n}(1-e^{-\alpha})}{e^{\alpha (n-1)}e^{\alpha n}}=\frac{(1-e^{-\alpha})}{e^{\alpha (n-1)}}<\frac{1}{e^{\alpha N}},$$
we get $e^{\alpha N}(1-e^{-\alpha})<e^{\alpha (n-1)}$. Hence, we have $e^{\alpha N}<\frac{e^{\alpha (n-1)}}{(1-e^{-\alpha})}$. By taking the logarithm, we obtain
$$n>\left\lfloor N+\frac{1}{\alpha}\log(1-e^{-\alpha})+1  \right\rfloor.$$

 Therefore, depending on $\alpha$, the number of singleton neighbourhoods is determined. In fact, if $\frac{1}{\alpha}\log(1-e^{-\alpha})\ge-2$, by Remark \ref{Remark: golden ration} for $\alpha> \gamma$, all neighbourhoods preceding $\mathcal{N}_{\rho_2}(N)$ are singleton. Otherwise, the number of singleton neighbourhoods is smaller and there are some finite neighbourhoods with at least two elements. Consequently, we consider two cases; the first case, $\alpha>\gamma$, and the second case, $\alpha\le \gamma$, and find the number of neighbourhoods before $\mathcal{N}_{\rho_2}(N)$ in each case. 
 
 First, consider $\alpha>\gamma$. In this case, $N-1$ is the largest number whose corresponding neighbourhood is a singleton, i.e. $\mathcal{N}_{\rho_2}(n)=\left\{n\right\},$ for all $n\le N-1$. We add all these neighbourhoods to $\mathcal{V}_\delta^{\alpha}$. Hence, $\mathcal{V}_\delta^{\alpha}$ is the optimal finite open cover for $\N$ with respect to ${\rho_2}$ for ${\alpha>\gamma}$ that includes $N-1$ singleton neighbourhoods and one infinite neighbourhood $\mathcal{N}_{\rho_2}(N)$ of $\N$. Therefore, in this case, the total number of open neighbourhoods in $\mathcal{V}_\delta^{\alpha}$ is equal to $K_\delta^{\alpha}:=N-1+1=N=\left\lceil \frac{1}{\alpha}\log\frac{1}{\delta}\right\rceil$.

 Now, for case $\alpha\le \gamma$, since $\left\lfloor N+\frac{1}{\alpha}\log(1-e^{-\alpha}) +1 \right\rfloor$ is the greatest natural number admitting a singleton neighbourhood, we denote this number by $N_s^{\alpha}$, i.e. $$N_s^{\alpha}:=\left\lfloor N+\frac{1}{\alpha}\log(1-e^{-\alpha}) +1 \right\rfloor= \left\lfloor \frac{1}{\alpha}\log\frac{1}{\delta}+\frac{1}{\alpha}\log(1-e^{-\alpha})+1  \right\rfloor.$$

We add all $\mathcal{N}_{\rho_2}(n)$, for natural numbers $n<N_s^{\alpha}$ to $\mathcal{V}_\delta^{\alpha}$. We are now looking for non-singleton $\delta$-neighbourhoods in this case. To do this, we first compute the neighbourhood associated with $N-1$, i.e. $$\mathcal{N}_{\rho_2}(N-1),$$ the last neighbourhood before $\mathcal{N}_{\rho_2}(N)$, whose elements $i$ all satisfy: 
$${\rho_2}(N-1-i,N-1)=\frac{1}{e^{\alpha(N-1-i)}}-\frac{1}{e^{\alpha(N-1)}}\le \frac{1}{e^{\alpha N}}.$$
We simplify it as follows;

\begin{align*}
\frac{1}{e^{\alpha(N-1-i)}}-\frac{1}{e^{\alpha(N-1)}}&= \frac{e^{\alpha(N-1)}-e^{\alpha(N-1-i)}}{e^{\alpha(N-1-i)}e^{\alpha(N-1)}}=\frac{e^{\alpha(N-1)}(1-e^{-\alpha i})}{e^{\alpha(N-1)}e^{-\alpha i}e^{\alpha N}e^{-\alpha }}\\&= \left(\frac{1-e^{-\alpha i}}{e^{-\alpha i}e^{-\alpha }}\right)\frac{1}{e^{\alpha N}} \le \frac{1}{e^{\alpha N}}.\end{align*}
So, to find appropriate $i$, we have the following condition
$$ 1-e^{-\alpha i}\le e^{-\alpha i}e^{-\alpha },$$
which implies that $i\le\frac{1}{\alpha}\log\left(e^{-\alpha}+1\right)$. So, we set $i=\left\lfloor\frac{1}{\alpha}\log\left(e^{-\alpha}+1\right)\right\rfloor$. Now, in view of Remark \ref{Remark: golden ration}, in this case, as $\alpha\le\gamma$, our $i$ can take large values, which implies that there are several elements in $\mathcal{N}_{\rho_2}(N-1)$. Now, we compute the first element of $\mathcal{N}_{\rho_2}(N-1)$ using $i$,

$$N-1-i= N-1-\left\lfloor\frac{1}{\alpha}\log\left(e^{-\alpha}+1\right)\right\rfloor.$$
Therefore, we obtain

$$\mathcal{N}_{\rho_2}(N-1)=\left\{N-C_{\alpha,1},\ldots,N-1\right\},$$

where $C_{\alpha,1}:=1+\left\lfloor\frac{1}{\alpha}\log\left(e^{-\alpha}+1\right)\right\rfloor$. We add this to $\mathcal{V}_\delta^{\alpha}$ and assume that there exist other neighbourhoods similar to $\mathcal{N}_{\rho_2}(N-1)$, that contain more than one element, then estimate the number of singleton neighbourhoods and finally determine the total number of neighbourhoods. 

We now look at the neighbourhood before $\mathcal{N}_{\rho_2}(N-1)$, that is $\mathcal{N}_{\rho_2}\left(N-C_{\alpha,1}-1\right)$ in which the elements satisfy
$${\rho_2}\left( N-2-\frac{1}{\alpha}\log\left(e^{-\alpha}+1\right)-i, N-2-\frac{1}{\alpha}\log\left(e^{-\alpha}+1\right)\right)\le\frac{1}{e^{\alpha N}},$$
which simplifies to
\begin{align}\label{Inequality alpha<gamma (1)}
    \frac{e^{\alpha\left(N-2-\frac{1}{\alpha}\log\left(e^{-\alpha}+1\right)\right)}-e^{\alpha\left(N-2-\frac{1}{\alpha}\log\left(e^{-\alpha}+1\right)-i\right)}}{e^{\alpha\left(N-2-\frac{1}{\alpha}\log\left(e^{-\alpha}+1\right)-i\right)}e^{\alpha\left(N-2-\frac{1}{\alpha}\log\left(e^{-\alpha}+1\right)\right)}}=\frac{1-e^{-\alpha i}}{e^{\alpha\left(N-2-\frac{1}{\alpha}\log\left(e^{-\alpha}+1\right)-i\right)}}\le \frac{1}{e^{\alpha N}}.
\end{align}
In order to find $i$, that gives the cardinality of $\mathcal{N}_{\rho_2}\left(N-C_{\alpha,1}-1\right)$, we have the following condition by \eqref{Inequality alpha<gamma (1)},
\begin{align*}
    \frac{1-e^{-\alpha i}}{e^{\alpha(-2-\frac{1}{\alpha}\log\left(e^{-\alpha}+1\right)-i)}}\le 1,
\end{align*}
which implies $$1-e^{-\alpha i}\le e^{-2\alpha-\log\left(e^{-\alpha}+1\right)}e^{-i\alpha},$$ from which it follows that $e^{\alpha i}\le e^{-2\alpha-\log\left(e^{-\alpha}+1\right)}+1 $. Hence,
$i\le \frac{1}{\alpha}\log \left(e^{-2\alpha-\log\left(e^{-\alpha}+1\right)}\right)$, and we set $i=\left\lfloor \frac{1}{\alpha}\log \left(e^{-2\alpha-\log\left(e^{-\alpha}+1\right)}\right)\right\rfloor$. So, the smallest number in $\mathcal{N}_{\rho_2}\left(N-C_{\alpha,1}-1\right)$ is 
\begin{align*}
 N-2-\left\lfloor\frac{1}{\alpha}\log\left(e^{-\alpha}+1\right)\right\rfloor-i&=  N-2-\left\lfloor\frac{1}{\alpha}\log\left(e^{-\alpha}+1\right)\right\rfloor- \left\lfloor\frac{1}{\alpha}\log \left(e^{-2\alpha-\log\left(e^{-\alpha}+1\right)}\right)\right\rfloor
\end{align*}
Hence,
\begin{align*}
&\mathcal{N}_{\rho_2}\left(N-C_{\alpha,1}-1\right) =\left\{  N-C_{\alpha,2},\ldots,N-C_{\alpha,1}-1 \right\},\end{align*}
where $C_{\alpha,2}:= 2+\left\lfloor\frac{1}{\alpha}\log\left(e^{-\alpha}+1\right)\right\rfloor+ \left\lfloor\frac{1}{\alpha}\log \left(e^{-2\alpha-\log\left(e^{-\alpha}+1\right)}\right)\right\rfloor$. Similarly, by continuing this process, the first elements of these neighbourhoods can be represented in the form of $N-C_{\alpha,j}$, where $C_{\alpha,j}$ is a constant in terms of $\alpha$ in the $j$-th such neighbourhood. So, to compute the number of these neighbourhoods, by comparing the largest number that admits a singleton neighbourhood, namely $N_s^{\alpha}$, with the smallest number in the first finite neighbourhood after singleton neighbourhoods, we can determine the number of non-singleton finite neighbourhoods. So, if we denote $j'$ the number of finite neighbourhoods that are not singleton, then we have
$$N-C_{\alpha,j'}-1= \left\lfloor N+\frac{1}{\alpha}\log(1-e^{-\alpha}) +1\right\rfloor,$$
which implies that $C_{\alpha,j'}=\left\lfloor\frac{-1}{\alpha}\log(1-e^{-\alpha})-2\right\rfloor$. Hence, $j'$, the number of these neighbourhoods, depends on $\alpha$ and is a constant, which we denote by $C_{\alpha}$. Now we also add all such neighbourhoods to $\mathcal{V}_\delta^{\alpha}$. Hence, $\mathcal{V}_\delta^{\alpha}$ is the optimal finite open cover for $\N$ with respect to ${\rho_2}$.

To find the number of open neighbourhoods in $\mathcal{V}_\delta^{\alpha}$, we have
\begin{align*}
K_\delta^{\alpha}=1+j'+N_s^{\alpha}=1+C_\alpha+ \left\lfloor N+\frac{1}{\alpha}\log(1-e^{-\alpha}) +1\right\rfloor= \left\lfloor \frac{1}{\alpha} \log \left(\frac{1}{\delta}\right)\right\rfloor+\mathcal{O}_\alpha(1).
\end{align*}
Hence, in both cases of $\alpha$, we obtain an optimal finite open cover $\mathcal{V}_\delta^{\alpha}$ with $K_\delta^{\alpha}= \left\lfloor \frac{1}{\alpha} \log \left(\frac{1}{\delta}\right)\right\rfloor+\mathcal{O}_\alpha(1)$ elements.

So, for $\delta>0$, there exists the finite set $$W_\delta^{\alpha}:=\left\{1,\ldots,N-1,N\right\},$$ with $K_\delta^{\alpha}$ elements, such that $\N\subset\bigcup_{i=1}^{K_\delta^{\alpha}}\mathcal{N}_{\rho_2}(v_i)$, where $v_i\in W_\delta^{\alpha}$. We set $V_i^{\alpha}:=\mathcal{N}_{\rho_2}(v_i)$ for $1\le i\le K_\delta^{\alpha} $, then \begin{equation}\label{Def:V_delta d_2+}
    \mathcal{V}_\delta^{\alpha}=\{V_i^{\alpha}\}_{i=1}^{K_\delta^{\alpha}}
\end{equation} is the open cover of $\N$, with respect to which ${\rho_2}$ is totally bounded.

We can now construct, using the collection $\mathcal{V}_\delta^{\alpha}$, a finite collection of cylinders for $\Sigma$. Let $k\ge 1$ and suppose $1\le i_0,\ldots,i_{k-1}\le K_\delta^{\alpha}$, then we consider the following sets.
\begin{equation}\label{Def: Z_i_0 d_2+}
(i_0,\ldots,i_{k-1})_{\mathcal{V}_\delta^{\alpha}}:=\left\{ (x_0,\ldots,x_{k-1},\ldots): x_0\in V_{i_0}^{\alpha},\ldots,x_{k-1}\in V_{i_{k-1}}^{\alpha}  \right\}.\end{equation}
Hence, these sets are a union of cylinders, i.e., $(i_0,\ldots,i_{k-1})_{\mathcal{V}_\delta^{\alpha}}=\bigcup_{\i\in\mathrm{U^*}}[\i]$, where $$\mathrm{U^*}=\left\{ (x_0,\ldots,x_{k-1}):  x_0\in V_{i_0}^{\alpha},\ldots,x_{k-1}\in V_{i_{k-1}}^{\alpha}    \right\}.$$ 

To calculate $\text{diam}\left((i_0,\ldots,i_{k-1})_{\mathcal{V}_\delta^{\alpha}}\right)$, with the analogous arguments for the diameter of $(i_0,\ldots,i_{k-1})_{\mathcal{V}_\delta}$ in Proposition \ref{Lemma:Totally bounded in Full shift}, the diameter of $(i_0,\ldots,i_{k-1})_{\mathcal{V}_\delta^{\alpha}}$ is determined by elements of different cylinders, i.e. $\ulx=(l_{i_0},\ldots,l_{i_{k-1}},1,\ldots)$ and $\uly=(s_{i_0},\ldots,s_{i_{k-1}},M,\ldots),$
where $M$ is a large number and for all $i_0\le i\le i_{k-1}$ if $V_i^\alpha\neq\mathcal{N}_{\rho_2}(N)$, $l_i,s_i\in V_i^\alpha$ are \st $\text{diam}(V_i^\alpha)={\rho_2}(l_i,s_i)$, and if $V_i^\alpha=\mathcal{N}_{\rho_2}(N)$, $s_i=N$ and $l_i=l$ \st $l\to\infty$. So, $$ d_2(\ulx,\underline{y})=\sum_{n=0}^{k-1} \theta^n \text{diam}(V_i^{\alpha})+\sum_{n\ge k} \theta^n \left| \frac{1}{e^\alpha} - \frac{1}{e^{M\alpha}} \right|=  \delta\left(\frac{1-\theta^k}{1-\theta}\right)+\frac{\theta^k}{e^\alpha(1-\theta)},$$
where $M\to\infty$. 
Hence, 
\begin{equation} \label{Diam_Union_Full d_2+}
\text{diam}\left((i_0,\ldots,i_{k-1})_{\mathcal{V}_\delta^{\alpha}}\right)=\delta\left(\frac{1-\theta^k}{1-\theta}\right)+\frac{\theta^k}{e^\alpha(1-\theta)}.
\end{equation} 

Now, we define $\U_{\delta}^{\alpha}:=\left\{(i_0,\ldots,i_{k-1})_{\mathcal{V}_\delta^{\alpha}}: 1\le i_l\le K_{\delta}^{\alpha}, \ \forall \ 0\le l\le k-1 \right\}$, as a finite collection of open neighbourhoods with $\#(\U_{\delta}^{\alpha})=(K_\delta^{\alpha})^k$ that covers $\Sigma$.

\textit{Step 2.} Here we show that $\U_{\delta}^{\alpha}$ satisfies conditions (1) and (2). By the first step, we are able to derive a finite collection $$\U_{\delta}^{\alpha}=\left\{(i_0,\ldots,i_{k-1})_{\mathcal{V}_\delta^{\alpha}}: 1\le i_l\le K_{\delta}^{\alpha}, \ \forall \ 0\le l\le k-1 \right\},$$ for every $k\ge 1$. To ensure that this open cover satisfies properties (1) and (2), we assume that $k=\left\lceil \frac{\log(\delta e^{\alpha}(1-\theta))}{\log(\theta)}\right\rceil$.  Note that $\U_{\delta}^{\alpha}$ is a pairwise disjoint cover for $\Sigma$.

We can rewrite $(i_0,\ldots,i_{k-1})_{\mathcal{V}_\delta^{\alpha}}$ as an element of $\U_{\delta}^{\alpha}$ in form $(i_0,\ldots,i_{k-1})_{\mathcal{V}_\delta^{\alpha}}=\bigcup_{\i\in\mathrm{U^*}}[\textbf{i}]$ where $\mathrm{U^*}=\left\{ (x_0,\ldots,x_{k-1}):  x_0\in V_{i_0}^{\alpha},\ldots,x_{k-1}\in V_{i_{k-1}}^{\alpha}    \right\}\subset \Sigma^*$ is a countable set such that $|\textbf{i}|= \left\lceil\frac{\log(\delta e^\alpha(1-\theta))}{\log(\theta)}\right\rceil$, which implies condition (1).

To obtain condition (2), we note that the same number $ 1<T=\left( \frac{1}{1-\theta} \right)+1<\infty,$ that established for condition (2) of $\U_\delta$ in the Proposition \ref{Lemma:Totally bounded in Full shift} remain valid for the collection $\U_{\delta}^{\alpha}$ as well.
Hence, the finite collection $\U_{\delta}^{\alpha}$ satisfies condition (2).

Observe that, since $k=\left\lceil \frac{\log(\delta e^{\alpha}(1-\theta))}{\log(\theta)}\right\rceil$ and $K_\delta^{\alpha}= \left\lfloor \frac{1}{\alpha} \log \left(\frac{1}{\delta}\right)\right\rfloor+\mathcal{O}_\alpha(1)$, the total number of elements in $\U_{\delta}^{\alpha}$ is equal to 
\begin{align*}
    \#(\U_{\delta}^{\alpha})=(K_\delta^{\alpha})^k= \left\lfloor\left( \frac{1}{\alpha} \log \left(\frac{1}{\delta}\right)+\mathcal{O}_\alpha(1)\right) ^ { \log_\theta{\delta e^\alpha(1-\theta)}}\right\rfloor.
\end{align*}

We note that $\U_{\delta}^{\alpha}$ is minimal. In fact, for any collection of cylinders of length less than $k$ in $\Sigma$, this collection cannot have the property (2). So, the desired conclusion of the proposition is obtained.
\end{proof}

We now have a statement analogous to Theorem \ref{Lemma: upper bound for E(hitting time)} for $\mathcal{U}_\delta^{\alpha}$. In fact, this theorem shows that this result is uniform with respect to the different metrics.
\begin{theorem}\label{Lemma: upper bound for E(hitting time) d_2}
Assume that $(\Sigma,\sigma)$ is the countable full shift equipped with the metric $d_2$ as in \eqref{Def:Metric d_2}. Let $\phi:\Sigma\to\R$ be a potential \st $\phi\big|_{[n]}\le -\kappa\log (n+1)$ for some $\kappa>1$. Suppose that $ U=(i_0,\ldots,i_{k-1})_{\mathcal{V}_\delta}^{\alpha}$ is an element of finite open cover $\mathcal{U}_\delta^{\alpha}$ as defined in Proposition \ref{Lemma:Totally bounded in Full shift d_2}. Then there exists $0<c<1<C<\infty$ such that $$ \frac{c}{\mu( U)} \le\E_\mu\left(\tau_{U}(\ulx)\right)\le \frac{C}{\mu( U)}.$$
\end{theorem}

\begin{proof}
The proof is analogous to the proof given in Theorem \ref{Lemma: upper bound for E(hitting time)}. In fact, since $\mathcal{U}_\delta^{\alpha}$ satisfies the $\psi$-mixing property (as explained in Remark \ref{Remark: thida}), the upper bound follows. For the lower bound, it suffices to show that $\mu(U)= \mathcal{O}(\eta(k))$ where $U$ is an element in $\mathcal{U}_\delta^{\alpha}$ and $\eta(k)$ is an exponentially decreasing to zero as $k\to 0$. To this end, let $U=(i_0,\ldots,i_{k-1})_{\mathcal{V}_\delta}^{\alpha}\in \mathcal{U}_\delta^{\alpha}$. We use the following inequality, which was obtained in \eqref{calculating_mu}:
$\mu(U)\le K \prod_{j=0}^{k-1} \#V_{i_j}^\alpha (\min_{n\in V_{i_j}^\alpha} (n)+1)^{-\kappa}.$ Now, since in the proof of Proposition \ref{Lemma:Totally bounded in Full shift d_2} the quantity of the sets $V_{i_j}^\alpha$ vary with respect to $\gamma$, we examine $\mu(U)$ in two cases $\alpha\le\gamma$ and $\alpha>\gamma$.

First, we consider $\alpha>\gamma$. In this case, since we have $N-1$ singletons and one infinite neighbourhood, we consider three cases. First, all the $V_{i_j}^\alpha$ are finite (singletons). Second, all the $V_{i_j}^\alpha$ are infinite. Finally, the third case is a combination of the previous two, namely, some are finite, and others are infinite.

 In first case, as all the $V_{i_j}^\alpha$ are singleton, we have 
$$ \#V_{i_j}^\alpha (\min_{n\in V_{i_j}^\alpha} (n)+1)^{-\kappa}=(\min_{n\in V_{i_j}^\alpha} (n)+1)^{-\kappa}\le 2^{-\kappa}.$$
Hence, 
$\mu(U)\le K \prod_{j=0}^{k-1} \#V_{i_j}^\alpha (\min_{n\in V_{i_j}^\alpha} (n)+1)^{-\kappa}\le K (2^{-\kappa})^k,$
 which tends to zero exponentially as $k\to \infty$. 
 For the case all the $V_{i_j}^\alpha$ are infinite i.e. for all $1\le j\le k-1$, $V_{i_j}^\alpha=\mathcal{N}_{\rho_2}(N)$, by the analogous method in \eqref{Ineq: Case Infin. of Proof of lwbd of E(hit)}, we see that $\mu(U)=\mathcal{O}(\eta_1(k))$, where $\eta_1(k)$ tends to zero as $k\to\infty$.
 For the third case, if we consider the number of finite $V_{i_j}^\alpha$ is $k'$ for some $1<k'<k$, so there are $k-k'$ infinite $V_{i_j}^\alpha$ such that all are $\mathcal{N}_{\rho_2}(N)$. In this case, we use the same method in \eqref{calculating_mu} to conclude
 \begin{align*}
    \mu(U)&\le K \prod_{j=0}^{k-1} \sum_{u_j\in V_{i_j}^\alpha} (u_j+1)^{-\kappa}\le K \prod_{j=0}^{k'-1} \#V_{i_j}^\alpha (\min_{n\in V_{i_j}^\alpha} (n)+1)^{-\kappa} \prod_{j=k'}^{k-1}\sum_{u_j= N+1} (u_j)^{-\kappa}\\& \le K (2^{-\kappa})^k \eta_1(k),
\end{align*}
 and if we consider $\eta_2(k):=(2^{-\kappa})^k \eta_1(k)$, then $\eta_2(k)\to 0$ exponentially as $k\to 0$. So, in the case $\alpha>\gamma$, for all $U\in \mathcal{U}_\delta^{\alpha}$, $\mu(U)=\mathcal{O}(\eta(k))$, where $\eta(k)$ has exponential behaviour decaying to 0. 

 Now, if we consider the second case, where $\alpha\le\gamma$, the only difference from the previous case is that there are some finite $V_{i_j}^\alpha$ that are not singletons. If we show that, when all $V_{i_j}^\alpha$ are finite, $\mu(U)$ decreases exponentially to zero, then the other cases follow as in Case $\alpha>\gamma$. Therefore, we assume that all $V_{i_j}^\alpha$ are finite and not singletons. Here, we recall that the number of elements in $V_{i_j}^\alpha$ is a constant depending on $\alpha$, which we denote here by $C_j(\alpha)$, and we denote the minimal element by $N-C_{\alpha,j}$ where $C_{\alpha,j}$ is a constant. Hence,
$$\mu(U)\le K \prod_{j=0}^{k-1} \#V_{i_j}^\alpha (\min_{n\in V_{i_j}^\alpha} (n)+1)^{-\kappa}\le K \prod_{j=0}^{k-1} C_j(\alpha)(N-C_{\alpha,j}+1)^{-\kappa}<1.$$
So, by considering $\eta_3(k):=\prod_{j=0}^{k-1} C_j(\alpha)(N-C_{\alpha,j}+1)^{-\kappa}$, $\eta_3(k)\to 0$ exponentially as $k\to \infty$. The other cases can be established analogously to Case $\alpha>\gamma$.

The remainder of the proof proceeds as in Theorem \ref{Lemma: upper bound for E(hitting time)}. Hence, the proof of the theorem is complete.
\end{proof}



\subsection{Proof of Theorem \ref{Theorem: main theorem 2}} We now turn to the proof of Theorem \ref{Theorem: main theorem 2}, which follows a similar approach to that of Theorem \ref{Theorem: main theorem}. 

\begin{proof}[ Proof of Theorem \ref{Theorem: main theorem 2}] The idea of the proof is entirely analogous to that of Theorem \ref{Theorem: main theorem}. The lower bound follows in the same way. For the upper bound, by Proposition \ref{Lemma:Totally bounded in Full shift d_2}, there exists a finite open cover $\U_{\delta}^{\alpha}$ for $(\Sigma,\sigma)$ with $$\#(\U_{\delta}^{\alpha})= \left\lfloor\left( \frac{1}{\alpha} \log \left(\frac{1}{\delta}\right)+\mathcal{O}_\alpha(1)\right) ^ { \log_\theta{\delta e^\alpha(1-\theta)}}\right\rfloor,$$
that is a subpartition. Therefore, by a similar argument, and using Lemma \ref{Lemma:subpartition}, we obtain the following inequality,
  $$ \E_\mu(\tau_\delta)\le  C\left(\left\lceil\frac{\log(\delta e^\alpha (1-\theta))}{\log\theta}\right\rceil+\max_{1\le i_0,\ldots,i_{k-1}\le K_{\delta}^{\alpha}}\E_\mu\left(\tau_{(i_0,\ldots,i_{k-1})_{\mathcal{V}_\delta^{\alpha}}}\right)\right)\left( \sum_{k=1}^{\#({\U_{\delta}^{\alpha}})}\frac{1}{k}\right).$$
We now obtain the following by the upper bound of Theorem \ref{Lemma: upper bound for E(hitting time) d_2},
  \begin{align}\label{Inequality: upper bound for E in the proof d_2+}
   \notag\E_\mu(\tau_\delta)&\le C\left(\left\lceil\frac{\log(\delta e^\alpha (1-\theta))}{\log\theta}\right\rceil+\max_{1\le i_0,\ldots,i_{k-1}\le K_{\delta}^{\alpha}}\frac{1}{\mu\left((i_0,\ldots,i_{k-1})_{\mathcal{V}_\delta^{\alpha}}\right)}\right)\log \left(\#(\mathcal{U}_\delta^{\alpha}) \right)\\& \le C\left(\left\lceil\frac{\log(\delta e^\alpha (1-\theta))}{\log\theta}\right\rceil+\frac{1}{M_{\mu}(\delta/2T)}\right)\log \left(\#({\U_\delta^{\alpha}}) \right).
   \end{align}
We now proceed to compute $\log \left(\#(\mathcal{U}_\delta^{\alpha}) \right)$,
\begin{align*}
    \log \left(\#(\mathcal{U}_\delta^{\alpha}) \right)&=\log \left(\left(\frac{1}{\alpha}\log\frac{1}{\delta}+\mathcal{O}_\alpha(1)\right) ^ { \frac{\log(\delta e^\alpha(1-\theta))}{\log(\theta)}}\right) \\&=\frac{\log\delta+\alpha+\log(1-\theta)}{\log \theta}\left(\log\frac{1}{\alpha}+\log\left(\log\frac{1}{\delta}\right)\right)+\frac{\log\delta+\alpha+\log(1-\theta)}{\log \theta}\left(\frac{\mathcal{O}_\alpha(1)}{\frac{1}{\alpha}\log\frac{1}{\delta}}\right) \\&\ 
    =C\left(\log\left(\frac{1}{\delta}\right)\log\left(\log\frac{1}{\delta}\right)+\alpha\mathcal{O}_\alpha(1)\right),
\end{align*}

for some constant $C>0$. Now, by using this in \eqref{Inequality: upper bound for E in the proof d_2+}, as $$\left\lceil\frac{\log(\delta e^\alpha (1-\theta))}{\log(\theta)}\right\rceil\log\left(\frac{1}{\delta}\right)\log\left(\log\frac{1}{\delta}\right)=\mathcal{O}\left(  \left(\log\frac{1}{\delta}\right)^2\left(\log\log\frac{1}{\delta}\right) \right),$$ we finally have,

$$\E_\mu(\tau_\delta)\le \frac{C}{M_{\mu}(\delta/2T)} \left(\log\frac{1}{\delta}\right)\left(\log\log\frac{1}{\delta}\right), $$
for some $C>0$.
\end{proof}


\section{Example}\label{Section: Examples}
In this section, we examine an example that illustrates the result we have obtained.
\textbf{Full shift with Bernoulli measure.} As an application of our results, we consider the full shift $(\Sigma,\sigma)$ equipped with the Bernoulli measure, which is defined as $\mu([x_0,\ldots,x_{n-1}])=\prod_{j=0}^{n-1} 2^{-x_j}$. This corresponds to the potential $\phi\big|_{[n]}=-n\log 2$ on $\Sigma$.  This system is also studied in \cite[Proposition 3.1]{Zhao} for a different metric. Our goal is to apply Theorem \ref{Theorem: main theorem} and Theorem \ref{Theorem: main theorem 2}. We begin with the metric $d_1$.

So, considering the metric $d_1$ which is defined in \eqref{Def:Metric d}. Since the Bernoulli measure is an example of a Gibbs measure, by Proposition \ref{Lemma:Totally bounded in Full shift}, for all $\delta>0$, we can construct $\U_\delta$ the same way here. So, by the definition of $(i_0,\ldots,i_{k-1})_{\mathcal{V}_{\delta}}\in \U_\delta$ and Bernoulli measure, we have 
\begin{align}\label{Eq: measure of U by Bernoulli}
 \notag \mu((i_0,\ldots,i_{k-1})_{\mathcal{V}_{\delta}})&=\mu\left(\bigcup_{u_j\in V_{i_j},\ 0\le j\le k-1} [u_0,\ldots,u_{k-1}]\right)= \sum_{u_j\in V_{i_j},\ 0\le j\le k-1}\mu([u_0,\ldots,u_{k-1}])\\&= \sum_{u_j\in V_{i_j},\ 0\le j\le k-1} \prod_{j=0}^{k-1} 2^{-u_j}.  
\end{align}
Now, we can find $\min_{(i_0,\ldots,i_{k-1})_{\mathcal{V}_{\delta}}\in \U_\delta}\mu((i_0,\ldots,i_{k-1})_{\mathcal{V}_{\delta}})$, which is obtained for $(i_0,\ldots,i_{k-1})_{\mathcal{V}_{\delta}}$ where for all $0\le j\le k-1$, $V_{i_j}=\mathcal{N}_{\rho_1}(N)$, which are defined in \eqref{Countable set of N}, where $N$ satisfies $\frac{1}{N}\le\delta$, and as $\delta<\frac{1}{N-1}$, we obtain $N=\left\lceil\frac{1}{\delta}\right\rceil$ . Then, by \eqref{Eq: measure of U by Bernoulli} for $(i_0,\ldots,i_{k-1})_{\mathcal{V}_{\delta}}$ \st  for all $0\le j\le k-1$, $V_{i_j}=\mathcal{N}_{\rho_1}(N)$ we have,
\begin{equation*}\label{approx: mu of N-tail in example}
    \mu((i_0,\ldots,i_{k-1})_{\mathcal{V}_{\delta}})=  \sum_{u_j\ge N}  2^{-ku_j}=\frac{\left(2^k\right)^{-N}}{1-2^{-k}}.
\end{equation*}
Recall that $k=\left\lceil \frac{\log(\delta(1-\theta))}{\log(\theta)}\right\rceil= B\log\left(\frac{1}{\delta}\right)+\mathcal{O}(1)$ for $\U_\delta$, where $B>0$ is a constant. As we have $\frac{1}{1-2^{-k}}=1+o(1)$, we conclude that 
$$\min_{(i_0,\ldots,i_{k-1})_{\mathcal{V}_{\delta}}\in \U_\delta}\mu((i_0,\ldots,i_{k-1})_{\mathcal{V}_{\delta}})= B2^{-\frac{1}{\delta}\log\frac{1}{\delta}}(1+o(1)).$$

Using property (2) of Proposition \ref{Lemma:Totally bounded in Full shift}, there exists $1< T<\infty$ \st  
$$ M_\mu(\delta)=\min_{\ulx\in \Sigma}\mu(B_{\delta}(\ulx))\le \min_{(i_0,\ldots,i_{k-1})_{\mathcal{V}_{\delta}}} \mu((i_0,\ldots,i_{k-1})_{\mathcal{V}_{\delta}})\le \min_{\ulx\in \Sigma}\mu(B_{T\delta}(\ulx))=M_\mu(T\delta).$$
Note that for each $r>0$, the infinite $\delta$-neighbourhood in $\U_{r\delta}$ is $\mathcal{N}_\rho(N)$ where $N$ is satisfied in $\frac{1}{N}\le r\delta$ and $k=\left\lceil \frac{\log(r\delta(1-\theta))}{\log(\theta)}\right\rceil $, which implies \begin{equation}\label{approx: min on epsilon}
\min_{(i_0,\ldots,i_{k-1})_{\mathcal{V}_{r\delta}}\in \U_{r\delta}}\mu((i_0,\ldots,i_{k-1})_{\mathcal{V}_{r\delta}})= B2^{-\frac{1}{r\delta}\log\frac{1}{\delta}}(1+o(1)).\end{equation}

In particular, if we consider $\varepsilon= \frac{1}{2T}$, for $U\in\U_{\frac{\varepsilon\delta}{T}}$, by \eqref{approx: min on epsilon}, we have $\min_{ U\in\U_{\frac{\varepsilon\delta}{T}}} \mu(U)= B2^{-\frac{T}{\varepsilon\delta}\log\frac{1}{\delta}}(1+o(1))$, and by property (2) of Proposition \ref{Lemma:Totally bounded in Full shift} for $\U_{\frac{\varepsilon\delta}{T}}$, we get $$ B2^{-\frac{T}{\varepsilon\delta}\log\frac{1}{\delta}}(1+o(1)) \le M_\mu\left(\varepsilon\delta\right) .$$
On the other hand, for $\U_{\frac{\delta}{\varepsilon}}$  by \eqref{approx: min on epsilon}, $\min_{U\in\U_{\frac{\delta}{\varepsilon}}}\mu(U)=  B2^{-\frac{\varepsilon}{\delta}\log\frac{1}{\delta}}(1+o(1))$, and then
$$ M_\mu\left(\frac{\delta}{\varepsilon}\right)\le B2^{-\frac{\varepsilon}{\delta}\log\frac{1}{\delta}}(1+o(1)).$$
Now, by Theorem \ref{Theorem: main theorem}, 
\begin{equation} \label{Result for example d_1}
    c 2^{\frac{\varepsilon}{\delta}\log\frac{1}{\delta}}\le \E_\mu(\tau_\delta)\le C 2^{\frac{T}{\varepsilon\delta}\log\frac{1}{\delta}} \left(\log\frac{1}{\delta}\right)^2,
\end{equation} 
for some $0<c<C<\infty$.

Now, we consider the metric $d_2$ as described in \eqref{Def:Metric d_2}. Let $\alpha>0$. Recall that, by Proposition \ref{Lemma:Totally bounded in Full shift d_2}, we have a finite open cover $\U_\delta^\alpha$ that includes an infinite $\delta$-neihbourhood i.e. $\mathcal{N}_{\rho_2}(N)$ which $N$ satisfies in $\left\lceil\frac{1}{\alpha}\log\frac{1}{\delta}\right\rceil= N$. Using an analogous argument as above for $d_1$, as in this case, $k=\left\lceil \frac{\log(e^\alpha\delta(1-\theta))}{\log(\theta)}\right\rceil=D\log\left(\frac{1}{\delta}\right)+\mathcal{O}(1)$ for some constant $D>0$, we then obtain
$$\min_{(i_0,\ldots,i_{k-1})_{\mathcal{V}_{\delta}^\alpha}\in \U_\delta^\alpha}\mu((i_0,\ldots,i_{k-1})_{\mathcal{V}_{\delta}^\alpha})=  D2^{-\left(\frac{1}{\alpha}\log\frac{1}{\delta}\right)\log\frac{1}{\delta}}(1+o(1)).$$

Note that for each $r>0$, we have $$ 
\min_{(i_0,\ldots,i_{k-1})_{\mathcal{V}_{r\delta}^\alpha}\in \U_{r\delta}^\alpha}\mu((i_0,\ldots,i_{k-1})_{\mathcal{V}_{r\delta}^\alpha})= D2^{-\left(\frac{1}{\alpha}\log\frac{1}{r\delta}\right)\log\frac{1}{\delta}}(1+o(1)).$$

Now, considering $\varepsilon=\frac{1}{2T}$ for $U\in\U_{\frac{\varepsilon\delta}{T}}^\alpha$ we have $\min_{ U\in\U_{\frac{\varepsilon\delta}{T}}^\alpha} \mu(U)= D2^{-\left(\frac{1}{\alpha}\log\frac{T}{\varepsilon\delta}\right)\log\frac{1}{\delta}}(1+o(1))$ and then $$D2^{-\left(\frac{1}{\alpha}\log\frac{T}{\varepsilon\delta}\right)\log\frac{1}{\delta}}(1+o(1))\le M_\mu\left(\varepsilon\delta\right).$$
Similarly for $\U_{\frac{\delta}{\varepsilon}}^\alpha$, we have $\min_{U\in\U_{\frac{\delta}{\varepsilon}}^\alpha}\mu(U)= D2^{-\left(\frac{1}{\alpha}\log\frac{\varepsilon}{\delta}\right)\log\frac{1}{\delta}}(1+o(1))$, and then
$$  M_\mu\left(\frac{\delta}{\varepsilon}\right)\le  D2^{-\left(\frac{1}{\alpha}\log\frac{\varepsilon}{\delta}\right)\log\frac{1}{\delta}}(1+o(1)).$$
 Hence, by Theorem \ref{Theorem: main theorem 2}, 
 \begin{equation}\label{Result for example d_2}
 c2^{\frac{1}{\alpha}\log\left(\frac{T}{\varepsilon\delta}\right)\log\frac{1}{\delta}}\le \E_{\mu}(\tau_\delta(\ulx))\le C 2^{\frac{1}{\alpha}\log\left(\frac{\varepsilon}{\delta}\right)\log\frac{1}{\delta}} \left(\log \frac{1}{\delta}\right)\left(\log\log \frac{1}{\delta}\right),\end{equation} 

for some $0<c<C<\infty$.

 As we mentioned earlier in Remark \ref{Remark: Conclusion after Main Thm}(2), in comparison to the results obtained for this example in \eqref{Result for example d_1} and \eqref{Result for example d_2} and those derived in \cite[Example 7.4]{MikeNatali}, where they found $2^{\sqrt{\frac{1}{c\delta}}}\le \E_{\mu}(\tau_\delta) $ on the interval map, this system can be covered faster than it is with respect to the metrics $d_1$ and $d_2$. Additionally, this system with respect to the metric $d_2$ has a faster covering time than with respect to $d_1$.


 \end{document}